\documentclass[12pt, a4paper]{amsart}

\usepackage{amsmath}
\usepackage{geometry} 
\usepackage{amsthm,graphics,tabularx,amssymb,shapepar, eucal,braket}
\usepackage{latexsym,amsfonts,amssymb,amsmath,amsxtra,bbm,color,mathrsfs}
\usepackage{amscd}
\usepackage{mathdots}
\usepackage{stackengine}
\stackMath

\allowdisplaybreaks[4]

\usepackage{hyperref}
\usepackage{pdfsync}
\usepackage[all]{xypic}

\usepackage{verbatim}

\newcommand{\BA}{{\mathbb {A}}}

\newcommand{\BC}{{\mathbb {C}}}

\newcommand{\BQ}{{\mathbb {Q}}}

\newcommand{\BZ}{{\mathbb {Z}}}

\newcommand{\CA}{{\mathcal {A}}}

\newcommand{\CE}{{\mathcal {E}}}
\newcommand{\CF}{{\mathcal {F}}}
\newcommand{\CG}{{\mathcal {G}}}

\newcommand{\CO}{{\mathcal {O}}}

\newcommand{\CR}{{\mathcal {R}}}

\newcommand{\CW}{{\mathcal {W}}}
\newcommand{\CX}{{\mathcal {X}}}

\newcommand{\RE}{{\mathrm {E}}}

\newcommand{\RG}{{\mathrm {G}}}
\newcommand{\RH}{{\mathrm {H}}}
\newcommand{\RI}{{\mathrm {I}}}

\newcommand{\RM}{{\mathrm {M}}}
\newcommand{\RN}{{\mathrm {N}}}
\newcommand{\RO}{{\mathrm {O}}}
\newcommand{\RP}{{\mathrm {P}}}

\newcommand{\RR}{{\mathrm {R}}}

\newcommand{\RU}{{\mathrm {U}}}

\newcommand{\RZ}{{\mathrm {Z}}}

\newcommand{\Aut}{{\mathrm{Aut}}}

\newcommand{\GL}{{\mathrm{GL}}}

\newcommand{\Hom}{{\mathrm{Hom}}}

\newcommand{\Ind}{{\mathrm{Ind}}}

\renewcommand{\Re}{{\mathrm{Re}}}

\newcommand{\rk}{{\mathrm{k}}}

\newcommand{\diag}{\operatorname{diag}}

\newcommand{\sgn}{\operatorname{sgn}}

\newcommand{\od}{\operatorname{d}}

\newcommand{\oL}{\operatorname{L}}

\newcommand{\oZ}{\operatorname{Z}}

\newcommand{\g}{\mathfrak g}

\renewcommand{\k}{\mathfrak k}

\newcommand{\p}{\mathfrak p}

\renewcommand{\a}{\mathfrak a}
\renewcommand{\b}{\mathfrak b}

\newcommand{\n}{\mathfrak n}

\renewcommand{\l}{\mathfrak l}

\newcommand{\s}{\mathfrak s}

\renewcommand{\rk}{\mathrm k}

\newcommand{\Z}{\mathbb{Z}}
\newcommand{\C}{\mathbb{C}}
\newcommand{\R}{\mathbb R}

\newcommand{\K}{\mathbb{K}}

\newcommand{\A}{\mathbb{A}}

\newcommand{\abs}[1]{\lvert#1\rvert}

\newcommand{\la}{\langle}
\newcommand{\ra}{\rangle}

\newcommand{\be}{\begin {equation}}
\newcommand{\ee}{\end {equation}}
\newcommand{\bee}{\begin {equation*}}
\newcommand{\eee}{\end {equation*}}

\newcommand{\cf}{\emph{cf.}~}

\theoremstyle{Theorem}

\theoremstyle{Theorem}

\newtheorem*{theoremA'}{Theorem A'}

\theoremstyle{Theorem}

\theoremstyle{Theorem}
\newtheorem{prp}{Proposition}[section]
\newtheorem{corp}[prp]{Corollary}
\newtheorem{lemp}[prp]{Lemma}
\newtheorem{thmp}[prp]{Theorem}

\theoremstyle{Plain}

\newtheorem{remarkp}[prp]{Remark}

\theoremstyle{Definition}

\newtheorem{dfnp}[prp]{Definition}

\numberwithin{equation}{section}

\begin{document}

	\title[Whittaker periods]{Betti-Whittaker periods of the contragredient representations for  $\GL(n)$}

	\author[Y. Jin]{Yubo Jin}
	\address{Institute for Advanced Study in Mathematics, Zhejiang University\\
		Hangzhou, 310058, China}\email{yubo.jin@zju.edu.cn}
	
	\author[D. Liu]{Dongwen Liu}
	\address{School of Mathematical Sciences,  Zhejiang University\\
		Hangzhou, 310058, China}\email{maliu@zju.edu.cn}

	\author[B. Sun]{Binyong Sun}
	\address{Institute for Advanced Study in Mathematics and  New Cornerstone Science Laboratory, Zhejiang University\\
		Hangzhou, 310058, China}\email{sunbinyong@zju.edu.cn}


	\maketitle
	\begin{abstract}
   We define Betti–Whittaker periods for a broad class of cohomological automorphic representations and establish a relation between the periods associated with these representations and their contragredients. This extends a result of Shih-Yu Chen for certain cuspidal automorphic representations. 
	\end{abstract}
	
	\tableofcontents
	\section{Introduction and the main result}


Shih-Yu Chen proved in \cite[Theorem~1.1]{C24} a period relation between the (bottom-degree) Betti--Whittaker periods attached to a cohomological cuspidal automorphic representation of $\GL_n$ and those of its contragredient, under certain regularity assumptions. The main goal of the present article is to generalize this result by removing those assumptions and extending it to a larger class of automorphic representations. Such a generalization is essential for applying the functional equation to study the rationality of special values of certain $L$-functions (see \cite{JLS26}).


  \subsection{Generic cohomological representations}\label{seccoh}

Let $\rk$ be a number field, and write $\A=\rk_\infty\times \A_\mathrm f$ for its adèle ring, where $\rk_\infty:=\rk\otimes_\BQ \R$ is the infinite part and $\A_\mathrm f$ is the finite part. 
 Fix a positive integer $n$, and let $\mathrm G$ denote the algebraic group $\GL_n$ defined over $\rk$. 

  Let $\CF$ denote the coefficient space defined in \eqref{defcf}. It is a locally algebraic representation of $\RG(\rk\otimes_\BQ \C)$ in which every irreducible algebraic representation  occurs with multiplicity one.
  
 As usual, an irreducible Casselman-Wallach representation $\pi$ of $\RG(\rk_\infty)$ is said to be cohomological if there exists an irreducible subrepresentation $F_\pi\subset \CF$ such that  the total continuous cohomology (see \cite{Clo, HM62, VZ84})
\be\label{totalcoh0}
\RH^*_{\rm ct}(\RG(\rk_\infty)^0; F_\pi\otimes \pi)\neq \{0\}. 
\ee
Here and henceforth, a superscript `$0$' over a Lie group indicates its identity connected component.
We call the subrepresentation $F_\pi$ the coefficient system of $\pi$. The determinant homomorphism yields an identification $\pi_0(\rk_\infty^\times)=\pi_0(\RG(\rk_\infty))$ ($\pi_0$ indicates the connected component group). Via this identification, the space in \eqref{totalcoh0}  is naturally  a representation of $\pi_0(\rk_\infty^\times)$. 
Fix a quadratic character $\varepsilon: \pi_0(\rk_\infty^\times) \rightarrow \C^\times$. 

Let $\CW_\infty$ be the Archimedean Whittaker space defined in \eqref{defarwh}. Let  $\CR^{\mathrm{coge},\infty}$ denote the set of all cohomological irreducible subrepresentations $\pi$ of $\CW_\infty$. The representations in $\CR^{\mathrm{coge},\infty}$ are well understood by Proposition \ref{class00} (see also \cite{BW, C26, Clo}).
Denote by $\CR_\varepsilon^{\mathrm{coge},\infty}$ the subset of all such $\pi$ for which $\varepsilon$ occurs in the space \eqref{totalcoh0}. 

Write 
\[
  b:=r_1\cdot \left\lfloor\frac{n^2}{4}\right\rfloor+r_2\cdot  \frac{n(n-1)}{2},
  \]
  where $r_1$ and $r_2$ denote the numbers of real places and complex places of $\rk$, respectively. 
For every $\pi\in \CR^{\mathrm{coge},\infty}$, denote by 
\[
  \RH^b(\pi):=\RH^b_{\rm ct}(\RG(\rk_\infty)^0; F_\pi\otimes \pi),
\]
which is the cohomology space of minimal non-vanishing degree. 
Denote by $\RH_\varepsilon^b(\pi)\subset \RH^b(\pi)$ the  $\varepsilon$-isotypic subspace (similar notation with a subscript `$\varepsilon$' will be used for other cohomology spaces without explanation). 
This space is one-dimensional if  $\pi\in \CR_\varepsilon^{\mathrm{coge},\infty}$, and otherwise it is zero. 

As a consequence of Delorme’s lemma (\cite[Theorem III.3.3]{BW}), we know that the map
\[
  \CR_\varepsilon^{\mathrm{coge},\infty}\rightarrow \{\textrm{irreducible subrepresentations of $\CF$}\}, \quad \pi\mapsto F_\pi
\]
is bijective. In particular, there exists a unique representation $\pi_\varepsilon\in \CR_\varepsilon^{\mathrm{coge},\infty}$ with trivial coefficient system. Fix a generator  
\[
0\neq \kappa_\varepsilon\in  \RH_\varepsilon^b(\pi_\varepsilon).
\]
For every $\pi\in \CR_\varepsilon^{\mathrm{coge},\infty}$, by using  Zuckerman-Jantzen translations we have a canonical isomorphism (see Lemma \ref{iotapi})
\[
\iota_\pi: \RH_\varepsilon^b(\pi_\varepsilon)\xrightarrow{\sim} \RH_\varepsilon^b(\pi).
\]
Write 
\[
0\neq \kappa_{\varepsilon,\pi}:=\iota_\pi(\kappa_\varepsilon)\in  \RH_\varepsilon^b(\pi).
\]

\subsection{Cohomological tamely isobaric automorphic representations}

For every reductive linear algebraic group $\RR$ over $\rk$, write \[
\CA(\RR):=\CA(\RR(\rk)\backslash \RR(\A)) \supset \CA_\mathrm{cusp}(\RR):=\CA_\mathrm{cusp}(\RR(\rk)\backslash \RR(\A))
\]
for the space of smooth automorphic forms on $\RR(\rk)\backslash \RR(\A)$ and its subspace of cusp forms, respectively. Both spaces carry canonical topologies and are smooth representations of $\RR(\A)$ (see \cite[Section 3.2]{LS} and \cite{G23}). For every irreducible subrepresentation $\Pi'$ of $\CA_\mathrm{cusp}(\GL_{n'})$ ($n'\geq 1$), its exponent $\mathrm{ep}(\Pi')\in \R$ is defined by requiring that    \be\label{defep}
\textrm{$\Pi'\otimes \abs{\det}_\A^{-\mathrm{ep}(\Pi')}\ $ is unitarizable.}
\ee 
Here and as usual, $\abs{\,\cdot\,}_\A: \A\rightarrow \R$ is the absolute value map.

 Fix a partition 
    \[
     n=n_1+n_2+\dots+n_r \qquad ( r, n_1,n_2,\dots\, n_r\geq 1).
    \]
 When confusion is not possible, we will not distinguish an irreducible representation from its isomorphism class. Let  
\[
\Pi:=\Sigma_{1}\boxplus  \Sigma_{2}\boxplus  \dots \boxplus \Sigma_{r} 
\]
be an isobaric automorphic representation of $\GL_n(\A)$ so that (see Section \ref{secisob})
\[
\Pi_v\cong \Sigma_{1,v}\boxplus  \Sigma_{2,v}\boxplus  \dots \boxplus \Sigma_{r,v}\quad \textrm{for all local places $v$ of $\rk$,}
\]
where  $\Sigma_i$ ($1\leq i\leq r$) is an irreducible cuspidal smooth automorphic representation of $\GL_{n_i}(\A)$. 
Here and henceforth, a subscript `$v$' indicates the local component at $v$. Suppose that  the exponents satisfy
		\[
		\mathrm{ep}(\Sigma_1)\geq\mathrm{ep}(\Sigma_2)\geq\dots\geq\mathrm{ep}(\Sigma_r).
		\]

Write 
    \[
     \RM=\GL_{n_1}\times \GL_{n_2}\times \dots \times \GL_{n_r}, 
     \]
    which is viewed as a Levi subgroup of $\RG$ as usual. Denote by $\RP$ the block upper-triangular parabolic subgroup of $\RG$ with Levi factor $\RM$. 
 Set 
 \[
 \Sigma=\Sigma_1\widehat{\otimes} \Sigma_2\widehat{\otimes} \dots\widehat{\otimes} \Sigma_r\subset\CA_{\mathrm{cusp}} (\RM),
 \]
where    $\widehat{\otimes} $ indicates the completed inductive tensor product (see \cite[Definition 43.5]{Tr}). Define the standard representation attached to $\Pi$ to be 
\[
 \mathrm I(\Pi):=\Ind_{\RP(\A)}^{\RG(\A)}\Sigma.
\]
Here and henceforth, `$\Ind$' indicates normalized smooth induction. Let $\rk_v$ denote  the local field  associated with a local place $v$ of $\rk$.  
Write
\[
\mathrm I(\Pi)=\mathrm I_\infty(\Pi)\otimes \mathrm I_\mathrm f(\Pi)=\widehat \otimes' \RI_v(\Pi),
\]
where 
\[
\mathrm I_\infty (\Pi):=\Ind_{\RP(\rk_\infty)}^{\RG(\rk_\infty)} \Sigma_\infty\quad 
\mathrm I_\mathrm f (\Pi):=\Ind_{\RP(\A_\mathrm f)}^{\RG(\A_\mathrm f)}\Sigma_\mathrm f , \quad \textrm{and}\quad \mathrm I_v (\Pi):=\Ind_{\RP(\rk_v)}^{\RG(\rk_v)}\Sigma_v. 
\]
Here $\Sigma_\infty$ and $\Sigma_\mathrm f$ respectively denote the infinite part and the finite part of $\Sigma$ (similar notation will be used without further explanation). 
By Theorem  \ref{prop:std}, $\RI_v(\Pi)$ is isomorphic to the standard module whose Langlands quotient is isomorphic to $\Pi_v$. 

\begin{dfnp}\label{def:isobaric}
	An automorphic representation  of $\RG(\A)$ is said to be  tamely isobaric if it is an isobaric automorphic representation $\Pi$ as above such that for all distinct $i,j\in \{1,2,\dots, r\}$ with $n_i=n_j$,  the following (to be called the tame condition) is satisfied: 
         \be\label{deftame0}
         \Sigma_i\neq  \Sigma_j\quad \textrm{and}\quad   \Sigma_i\neq \Sigma_j\otimes \abs{\det}_\A.
     \ee
	\end{dfnp}
    We call the first inequality in \eqref{deftame0} the first tame condition, and the second one the second tame condition. 

As usual, a subrepresentation of $\CA(\RG)$ is said to be globally generic if  the   global Whittaker integral is not identically zero on it.
When $\Pi$ is tamely isobaric, Eisenstein series associated to $\RI(\Pi)$ yield a  globally generic subrepresentation $\RI^{\mathrm{aut}}(\Pi)$ of $\CA(\RG)$. It depends only on $\Pi$,  and is isomorphic to $\RI(\Pi)$ (see Proposition \ref{lemti0} and Remark \ref{iaut}). 
 
 Let $\CR^{\mathrm{coti}}$ denote the set of isomorphism classes of tamely isobaric automorphic representations $\Pi$ of $\RG(\A)$ such that 
\be\label{coti}
     \textrm{ $\Pi_\infty$ is generic and cohomological.}
     \ee
Note that $\Pi_\infty$ is generic if and only if    $\RI_\infty(\Pi)$ is irreducible, in which case $\Pi_\infty\cong \RI_\infty(\Pi)$. 
  By Proposition \ref{prpactpi}, we have a natural action $\Aut(\C) \curvearrowright \CR^{\mathrm{coti}}$. Moreover, 
  \cite[Conjecture 3.7]{Clo} holds for every representation $\Pi \in \CR^{\mathrm{coti}}$, namely its stabilizer group is open in $\Aut(\C)$ (see Section \ref{secsmac} for the topology on $\Aut(\C)$).

 Let   $\CR_\varepsilon^{\mathrm{coti}}\subset \CR^{\mathrm{coti}}$ denote the subset of all the representations whose infinite part is isomorphic to a representation in   $\CR_\varepsilon^{\mathrm{coge},\infty}$. This subset is $\Aut(\C)$-stable. 
 
 


\subsection{$\Aut(\C)$-actions}

Let $K_\infty$ denote the standard maximal compact subgroup of $\RG(\rk_\infty)$, which is a product of compact orthogonal groups and compact unitary groups. 
 Define 
  \be \label{CXn}
  \CX:=\RG(\rk)\backslash \RG(\BA)/ K_{\infty}^0.
    \ee
    For every open compact subgroup $K_\mathrm f$ of $\RG(\BA_\mathrm f)$, the locally algebraic representation $\CF$ defines a sheaf on $\CX/K_\mathrm f$, which is still denoted by $\CF$. 
   Define
\be\label{sheafcoh}
\RH^{b}(\CX):= 
\varinjlim_{K_\mathrm f}\RH^{b}(\CX/ K_\mathrm f,\CF),
\ee
where $K_\mathrm f$ runs over the directed system of open compact subgroups of $\RG(\BA_\mathrm f)$, and $\RH^{b}(\CX/ K_\mathrm f,\CF)$ is the usual sheaf cohomology space. 

Suppose that $\Pi$ is in $\CR^{\mathrm{coti}}$. 
Write  $F_\Pi :=F_{\RI_\infty(\Pi)}\subset \CF$ for   the coefficient system of $\RI_\infty(\Pi)\cong \Pi_\infty$. Set 
\[
 \RH^b(\RI^{\mathrm{aut}}(\Pi)):=\RH^b_{\rm ct}(\RG(\rk_\infty)^0; F_\Pi\otimes \RI^{\mathrm{aut}}(\Pi)).
 \]

By realizing the continuous cohomology  as the relative Lie algebra cohomology, and  the sheaf cohomology as de Rham cohomology, the inclusion $F_\Pi\hookrightarrow \CF$ yields a linear map
\[
\iota_\Pi: \RH^b(\RI^{\mathrm{aut}}(\Pi))\rightarrow \RH^{b}(\CX). 
\]
This map is injective (see Lemma \ref{injbetti00}). 

 The usual action  $\Aut(\C)\curvearrowright\CF$ (see \eqref{sigmaf0}) yields an action 
 $\Aut(\C)\curvearrowright \RH^{b}(\CX)$. By Lemma \ref{actbetti}, this  further induces an action 
    \be\label{actbetti0000}
    \Aut(\C)\curvearrowright\bigsqcup_{\Pi\in  \CR_\varepsilon^{\mathrm{coti}} } \RH_\varepsilon^b(\RI^{\mathrm{aut}}(\Pi)).
 \ee

    Let $\CW_\mathrm f$ be the non-Archimedean Whittaker space defined in \eqref{defarwh2}, which carries a well-known action of $\Aut(\C)$ (see Section \ref{secsesq}). This action, together with the action \eqref{actbetti0000}, yields an action 
    \be\label{actbetti001}
\Aut(\C)\curvearrowright\bigsqcup_{\Pi\in  \CR_\varepsilon^{\mathrm{coti}} } \Hom_{\RG(\A_\mathrm f)}(\RH_\varepsilon^b(\RI^{\mathrm{aut}}(\Pi)),\CW_\mathrm f).
    \ee
 Lemma \ref{lem20} implies that the bundle in \eqref{actbetti001} is a line bundle, and by Theorem \ref{smooth2}, the action \eqref{actbetti001} is smooth sesquilinear in the sense of Definition \ref{defsmooth}.


 Write $\RI_\infty^{\mathrm{whi}}(\Pi)\subset \CW_\infty$ for the realization of the representation  $\RI_\infty(\Pi)$ in the Archimedean Whittaker space. Set 
 \[
   \RH^b(\RI_\infty^{\mathrm{whi}}(\Pi)):=\RH^b_{\rm ct}(\RG(\rk_\infty)^0; F_\Pi\otimes \RI_\infty^{\mathrm{whi}}(\Pi)).
 \]
 Whittaker period integrals yield a homomorphism (see Lemma \ref{jmathpi}) 
 \[
  \jmath:  \RI^{\mathrm{aut}}(\Pi)\rightarrow \RI_\infty^{\mathrm{whi}}(\Pi)\otimes \CW_{\rm f}, 
 \]
 which further induces an isomorphism 
 \[
 \jmath: \RH^b_\varepsilon(\RI_\infty^{\mathrm{whi}}(\Pi))^\vee\xrightarrow{\sim} \Hom_{\RG(\A_\mathrm f)}(\RH_\varepsilon^b(\RI^{\mathrm{aut}}(\Pi)),\CW_\mathrm f).
 \]
 Here the superscript `$^\vee$' indicates the dual space. Using these isomorphisms for various $\Pi$, the action \eqref{actbetti001} induces a smooth sesquilinear action 
    \be\label{actbetti005}
\Aut(\C)\curvearrowright\bigsqcup_{\Pi\in  \CR_\varepsilon^{\mathrm{coti}} } \RH^b_\varepsilon(\RI_\infty^{\mathrm{whi}}(\Pi))^\vee.
    \ee

\subsection{Betti-Whittaker periods}

Using Lemma \ref{sectionpi00}, we fix an $\Aut(\C)$-equivariant section
\[
\varpi: \CR_\varepsilon^{\mathrm{coti}}\rightarrow  \bigsqcup_{\Pi\in  \CR_\varepsilon^{\mathrm{coti}} } \left(\RH^b_\varepsilon(\RI_\infty^{\mathrm{whi}}(\Pi))^\vee\setminus \{0\}\right)
\]
of the natural map
\[
\bigsqcup_{\Pi\in  \CR_\varepsilon^{\mathrm{coti}} }  \left(\RH^b_\varepsilon(\RI_\infty^{\mathrm{whi}}(\Pi))^\vee\setminus \{0\}\right)\rightarrow  \CR_\varepsilon^{\mathrm{coti}}.  
\]

Write $\kappa_{\varepsilon,\Pi}:=\kappa_{\varepsilon,\RI^{\mathrm{whi}}_\infty(\Pi)}$ for simplicity. We give the following definition after \cite[Section 3.3, 3.4]{Mah} and \cite{RS08} (see also \cite{H, Hi94} for $n=2$). 

  \begin{dfnp}
    For each $\Pi\in \CR_\varepsilon^{\mathrm{coti}}$, its  (bottom degree) Betti-Whittaker period is defined to be
    \[
   \Omega_\varepsilon(\Pi) := \Omega_\varepsilon(\Pi; \varpi, \kappa_\varepsilon):=\la \varpi(\Pi), \kappa_{\varepsilon,\Pi}\ra\in \C^\times. 
    \]
\end{dfnp}

The set $\CR_\varepsilon^{\mathrm{coti}}$ is stable under taking the contragredient (\cf \eqref{mvwoncoti}).  Let $\psi:\rk\backslash\A\to\C^{\times}$ be a non-trivial character.  By using MVW-involutions (\cite{MVW87}), in Section \ref{sec:mvwarch} we will define a line bundle  involution
\[
\imath_\psi: \bigsqcup_{\Pi\in  \CR_\varepsilon^{\mathrm{coti}} }\RH^b_\varepsilon(\RI_\infty^{\mathrm{whi}}(\Pi))^\vee\rightarrow \bigsqcup_{\Pi\in  \CR_\varepsilon^{\mathrm{coti}} }\RH^b_\varepsilon(\RI_\infty^{\mathrm{whi}}(\Pi))^\vee
\]
over the base set involution 
\[
\CR_\varepsilon^{\mathrm{coti}}\rightarrow\CR_\varepsilon^{\mathrm{coti}}, \qquad \Pi\mapsto \Pi^\vee. 
\]
Here $\Pi^\vee$ denotes the contragredient of $\Pi$.  
By Corollary \ref{mvwkappa}, $\imath_\psi$ is determined by the following property: 
for all $\Pi\in \CR_\varepsilon^{\mathrm{coti}}$, the diagram 
\[
\begin{CD}
\RH^b_\varepsilon(\RI_\infty^{\mathrm{whi}}(\Pi))^\vee@>\imath_\psi >>\RH^b_\varepsilon(\RI_\infty^{\mathrm{whi}}(\Pi^\vee))^\vee\\
    @V \textrm{transpose of $\iota_{\RI_\infty^{\mathrm{whi}}(\Pi)}$}  VV @VV  \textrm{transpose of $\iota_{\RI_\infty^{\mathrm{whi}}(\Pi^\vee)}$}  V\\
\RH^b_\varepsilon(\pi_\varepsilon)^\vee @>\textrm{multiplication by }\delta_n >>\RH^b_\varepsilon(\pi_\varepsilon)^\vee
\end{CD}
\]
commutes. Here the isomorphisms $\iota_{\RI_\infty^{\mathrm{whi}}(\Pi)}$ and $\iota_{\RI_\infty^{\mathrm{whi}}(\Pi^\vee)}$ are defined in \eqref{iotapi2} by using Zuckerman-Jantzen translations, and
\be \label{delta000}
\delta_n : = (-1)^{r_1\cdot\frac{m(m+1)}{2}},\quad \text{where } m:=\left\lfloor \frac{n-1}{2}\right\rfloor
\ee
(recall that  $r_1$ is the number of real places of $\rk$). In particular, the involution $\imath_\psi$ is independent of $\psi$. 

Denote by $\omega_{\Pi}$ the central character of $\Pi$. 
Now we define another section 
\begin{eqnarray*}
\breve \varpi_\psi  :\CR_\varepsilon^{\mathrm{coti}}&\rightarrow  &\bigsqcup_{\Pi\in  \CR_\varepsilon^{\mathrm{coti}} } \left(\RH^b_\varepsilon(\RI_\infty^{\mathrm{whi}}(\Pi))^\vee\setminus \{0\}\right),\\
 \Pi^\vee &\mapsto &\mathscr G(\omega_{\Pi}, \psi)^{1-n}\cdot \imath_\psi(\varpi({\Pi})),
\end{eqnarray*}
where $\mathscr G(\omega_{\Pi}, \psi)$ is the Gauss sum (see \eqref{gausssum}). 

The following is the main result of the present paper, which  generalizes \cite[Theorem 1.1]{C24}.  
As this paper was nearing completion, we became aware that Castellano, Chen, Darshan, and Raghuram (\cite{CCDR}) had recently and independently proved related period relations for Betti–Whittaker periods under duality for cuspidal automorphic representations, together with variants and applications to critical values of L-functions.

\begin{thmp}\label{maintheorem}
The section $\breve \varpi_\psi$ is  $\Aut(\C)$-equivariant, and  
\[
\Omega_\varepsilon( \Pi^\vee; \breve \varpi_\psi, \delta_n\cdot \kappa_\varepsilon)=\mathscr G(\omega_{\Pi}, \psi)^{1-n}\cdot \Omega_\varepsilon(\Pi;\varpi, \kappa_\varepsilon)
\]
 for every $\Pi\in \CR_\varepsilon^{\mathrm{coti}}$. 
 Moreover, the map
  \be\label{auteq0}
\CR_\varepsilon^{\mathrm{coti}}\rightarrow \C^\times,\quad \Pi\mapsto \frac{\Omega_\varepsilon(\Pi^\vee;  \varpi, \kappa_\varepsilon)}{\mathscr{G}(\omega_{\Pi},\psi)^{1-n}\cdot \Omega_\varepsilon(\Pi; \varpi, \kappa_\varepsilon)}
  \ee
  is $\Aut(\C)$-equivariant.
\end{thmp}

Note that the quotient in \eqref{auteq0} is independent of $\varpi$ and $\kappa_\varepsilon$.

In addition to Theorem \ref{maintheorem}, several other results established herein are of independent interest. See Theorems \ref{prop:std}, \ref{lemti00}, and \ref{mvwpiK}.



\section{Preliminaries}

\subsection{Whittaker spaces and coefficient spaces}
Let $\mathrm N$ denote the algebraic subgroup of $\RG$ consisting of upper-triangular unipotent matrices. 
Recall the non-trivial character $\psi:\rk\backslash\A\to\C^{\times}$, which further induces an automorphic character
	\begin{equation}\label{psin}
		\begin{aligned}
			\psi_{\RN}:\RN(\rk)\backslash\RN(\A)&\to\C^{\times},\\
			[x_{i,j}]_{1\leq i,j\leq n}&\mapsto\psi\left(\sum_{i=1}^{n-1} x_{i,i+1}\right).
		\end{aligned}
	\end{equation}

Let 
$
\psi': \rk\backslash \A\rightarrow \C^\times
$
be another non-trivial character. Then there exists a unique element $a\in \rk^\times$ such that
\[
\psi'(x)=\psi(a^{-1}x)\quad \textrm{for all } x\in \A. 
\]
Define 
     \be\label{ta}
\mathbf{t}_a:={\rm diag}(a^{n-1}, \ldots, a, 1) \in \RG(\rk)
\ee
so that
\[
\psi'_{\RN}( u )=\psi_{\RN}(\mathbf t_a^{-1} u \mathbf t_a)\quad \textrm{for all } u\in \RN(\A). \]
Then we have a $\RG(\A)$-intertwining isomorphism 
\be\label{isopsi00}
  L_a: \Ind_{\RN(\A)}^{\RG(\A)}\psi_{\RN}\rightarrow \Ind_{\RN(\A)}^{\RG(\A)}\psi'_{\RN}, \qquad f\mapsto (x\mapsto f(\mathbf t_a^{-1} x)). 
\ee
Using the isomorphisms as in \eqref{isopsi00}, we form the representation 
\[
  \CW:=\varprojlim_{\psi} \Ind_{\RN(\A)}^{\RG(\A)}\psi_{\RN},
\]
to be called the global Whittaker space. Here $\psi$ runs through all non-trivial characters on $\rk\backslash \A$. 

Similarly to \eqref{isopsi00}, we have isomorphisms
\[
  L_a: \Ind_{\RN(\rk_\infty)}^{\RG(\rk_\infty)}\psi_{\RN,\infty}\rightarrow \Ind_{\RN(\rk_\infty)}^{\RG(\rk_\infty)}\psi'_{\RN,\infty}, 
\]
and \[
  L_a: \Ind_{\RN(\A_\mathrm f)}^{\RG(\A_\mathrm f)}\psi_{\RN,\mathrm f}\rightarrow \Ind_{\RN(\A_\mathrm f)}^{\RG(\A_\mathrm f)}\psi'_{\RN, \mathrm f}.
\]
 Using these isomorphisms 
we form the Archimedean Whittaker space 
\be\label{defarwh}
\CW_\infty:=\varprojlim_{\psi}\Ind_{\RN(\rk_\infty)}^{\RG(\rk_\infty)}\psi_{\RN,\infty}
\ee
and the non-Archimedean Whittaker space
\be\label{defarwh2}
\CW_\mathrm f:=\varprojlim_{\psi} \Ind_{\RN(\A_\mathrm f)}^{\RG(\A_\mathrm f)}\psi_{\RN,\mathrm f}.
\ee

We also have a $\RG(\rk\otimes_\BQ \C)$-intertwining isomorphism 
\be\label{isopsi03}
  L_a: \,^{\mathrm{alg}}\Ind_{\RN(\rk\otimes_\BQ \C)}^{\RG(\rk\otimes_\BQ \C)}\C\rightarrow \,^{\mathrm{alg}}\Ind_{\RN(\rk\otimes_\BQ \C)}^{\RG(\rk\otimes_\BQ \C)}\C, \qquad f\mapsto (x\mapsto f(\mathbf t_a^{-1} x)). 
\ee
Here and henceforth, `$\,^{\mathrm{alg}}\Ind$' stands for algebraic induction. Using the isomorphisms as in \eqref{isopsi03}, we form the representation 
\be\label{defcf}
  \CF:=\varprojlim_{\psi} \,^{\mathrm{alg}}\Ind_{\RN(\rk\otimes_\BQ \C)}^{\RG(\rk\otimes_\BQ \C)}\C,
\ee
to be called the coefficient space. Note that as a representation of $\RG(\rk\otimes_\BQ \C)$, $\CF$ is completely reducible, and every irreducible algebraic representation occurs exactly once in $\CF$.

\subsection{Whittaker periods}
Consider  the Whittaker period 
\be\label{generic}
\lambda_{\psi} : \CA(\RG)\rightarrow \C, \quad \phi\mapsto \int_{\mathrm{N}(\rk)\backslash \mathrm{N}(\A)} \phi(x) \psi_{\RN}^{-1}(x) \od\!  x,
\ee
 where $\od \! x$ is the $\mathrm{N}(\A)$-invariant Borel measure with total volume $1$.  We define   $\lambda_\psi: \CW\rightarrow \C$
 to be the composition of 
\be\label{lpsi00}
\CW\xrightarrow{\textrm{projection}}\Ind_{\RN(\A)}^{\RG(\A)}\psi_{\RN} \xrightarrow{\textrm{evaluation at the identity}}\C.
\ee

\begin{lemp}\label{jmathpi}
   There exists a unique homomorphism $$\jmath\in \Hom_{\RG(\A)}(\CA(\RG), \CW)$$ such that the diagram 
  \[
\xymatrix{
 \CA(\RG)  \ar[rrd]_{\lambda_{\psi}} \ar[rr]^{\qquad\jmath}&& \CW \ar[d]^{ \lambda_\psi} \\
& &  \C  \\
    }
\]
commutes for all non-trivial characters  $\psi: \rk \backslash \A\rightarrow \C^\times$. 
    \end{lemp}
\begin{proof}
Note that the equality
   \[\mathbf t_a. \lambda_\psi=\lambda_{\psi'}
    \]
    holds in both $\Hom_\C(\CA(\RG),\C)$ and 
$\Hom_\C(\CW,\C)$. This implies the lemma. 
\end{proof}

Let 
\be\label{jmath}
\jmath: \CA(\RG)\rightarrow  \CW
\ee
be as in Lemma \ref{jmathpi}. 

\subsection{Generic cohomological representations}\label{secgc}
We let  $\K = \R$ or $\C$ be an Archimedean local field. 
In this subsection, we formulate a classification of generic irreducible Casselman-Wallach representations $\pi$ of ${\rm GL}_n(\K)$ that are cohomological in the following sense: there exists an irreducible algebraic representation $F_\pi$ of $\GL_n(\K\otimes_\R \C)$  such that  the total continuous cohomology 
\be\label{totalcoh3}
\RH^*_{\rm ct}(\GL_n(\K)^0; F_\pi\otimes \pi)\neq \{0\}. 
\ee
Let $\CR^{\rm coge,\K}$ denote the set of all isomorphism classes of such representations.  


For $a, b\in \BC$ with $a-b\in \Z$, define a character 
\[
\chi_{a,b}: \C^\times \to \C^\times, \quad z\mapsto z^a\bar z^b := (z\bar z)^a \bar z^{b-a}.
\]
If moreover $a\neq b$, then we denote by $D_{a,b}$ the essentially square-integrable irreducible Casselman–Wallach representation of $\GL_2(\R)$ with infinitesimal character $(a,b)$. 

\begin{dfnp} 
A multiset $\gamma =\set{(a_1,b_1), (a_2, b_2), \ldots, (a_n, b_n)}$ of $n$ elements in $\left(\Z+\frac{n-1}{2}\right)^2$ is called 
incomparable if 
\[
(a_i-a_j)(b_i-b_j)<0 \quad \text{for all distinct }i, j =1,2,\ldots,n.
\]
It is called symmetric if 
\[
\gamma = \set{(b_1, a_1), (b_2, a_2), \ldots, (b_n, a_n)}.
\]
\end{dfnp}

In particular, an incomparable multiset $\gamma$ as above is a set. Denote
\[
\Gamma_n^\R:=
   \left\{\textrm{incomparable symmetric multisets of $n$ elements in $\left(\Z+\frac{n-1}{2}\right)^2$}\right\}
    \]
    and
 \[
   \Gamma_n^\C:=
\left\{\textrm{incomparable multisets of $n$ elements in $\left(\Z+\frac{n-1}{2}\right)^2$}\right\}.
\]
Note that for each $\gamma\in \Gamma_n^\R$, if $n$ is odd then  \be\label{agamma}
\textrm{there exists a unique $a_\gamma\in \Z+\frac{n-1}{2}$ such that $(a_\gamma,a_\gamma)\in \gamma$};
\ee
and if $n$ is even then no such $a_\gamma$ exists. 
 

Let $\mathscr K(\GL_n(\K))$ denote the Grothendieck group (with coefficients in $\Z$) of the category of Casselman-Wallach representations of $\GL_n(\K)$. An element of $\mathscr K(\GL_n(\K))$ is called a virtual representation of $\GL_n(\K)$. For every partition 
\be\label{parti}
n=n_1+n_2+\dots +n_r,\qquad (r, n_1, n_2, \dots n_r\geq 1),
\ee
the normalized smooth parabolic induction  yields a map 
\[
\begin{array}{rcl}
     \prod_{i=1}^r\mathscr K(\GL_{n_i}(\K))&\rightarrow &\mathscr K(\GL_n(\K)),   \\
\{\pi_i\}_{i=1,2,\dots,r}&\mapsto & \times_{i=1}^r \pi_i:=\pi_1\times \pi_2\times \dots \times \pi_r.
\end{array}
\]
Note that the virtual representation  $\times_{i=1}^r \pi_i$ remains unchanged when the $r$-tuple $(\pi_1, \pi_2, \dots, \pi_r)$ is replaced by a permutation of it.  
For
$
\gamma  \in \Gamma_n^\K,
$
define a virtual representation $\pi(\gamma)$ of $\GL_n(\K)$ by 
\be \label{pigamma}
\pi(\gamma):= \begin{cases}\times_{(a,b)\in \gamma} \chi_{(a,b)}, & \text{if $\K= \C$},\\
\times_{(a,b)\in \gamma_{1/2}} D_{a, b}, & \text{if $\K= \R$ and $n$ is even,}\\
(\times_{(a,b)\in \gamma_{1/2}} D_{a, b})\times ((\,\cdot\,)^{a_\gamma}\sgn^{\frac{1-n}{2}}), & \text{if $\K= \R$ and $n$ is odd},
\end{cases}
\ee
where `$\sgn$' denotes the sign character of $\R^\times$, $a_\gamma$ is as in \eqref{agamma}, and  $\gamma_{1/2}$ is a subset of $\gamma$ of $\lfloor \frac{n}{2}\rfloor$ elements such that
\[
  \gamma=\begin{cases}
      \gamma_{1/2}\cup \{(b,a)\,:\, (a,b)\in \gamma_{1/2}\}, & \text{if $n$ is even};\\
      \gamma_{1/2}\cup \{(b,a)\,:\, (a,b)\in \gamma_{1/2}\}\cup\{(a_\gamma,a_\gamma)\}, & \text{if $n$ is odd}.
  \end{cases}
\]

Let $\varepsilon_\K: \pi_0(\K^\times)\rightarrow \C^\times$ be a character, which is obviously identified with a quadratic character of $\K^\times$. 
Let $\CR_{\varepsilon_\K}^{\rm coge,\K}$ denote the set of all $\pi\in \CR^{\rm coge,\K}$ such that $\varepsilon_\K$ occurs in the total continuous cohomology \eqref{totalcoh3}. The sets $\CR_{\varepsilon_\K}^{\rm coge,\K}$ and $\CR^{\rm coge,\K}$ are obviously viewed as subsets of $\mathscr K(\GL_n(\K))$. As usual, every character of $\K^\times$ is  viewed as a character of $\GL_n(\K)$ by pullback through the determinant homomorphism (similar convention will be used for other general linear groups without explanation).

\begin{prp} \label{class00}
    The map 
    \[
    \Gamma_n^\K \to \CR^{\rm coge,\K}_{\varepsilon_\K},\quad \gamma \mapsto \pi(\gamma) \otimes \varepsilon_\K 
    \]
    is  well-defined and bijective. 
\end{prp}
\begin{proof}
    This is a consequence of Delorme's lemma (\cite[Theorem III.3.3]{BW}), and can be proved by similar calculations as in the essentially tempered case (see \cite[Lemma 3.14]{Clo}). 
\end{proof}
    
Suppose that we are given a partition as in \eqref{parti}. Denote
\be \label{nu_i}
\nu_i:=\frac{n-n_i}{2}-\sum_{j=1}^{i-1} n_j,\qquad\text{for }i=1,2,\dots,r.
\ee
 For every 
\be\label{gammaip}
\gamma_i':=\gamma_i-\nu_i:=\{(a-\nu_i, b-\nu_i)\,:\,{(a,b)\in \gamma_i}\},
\ee
where $\gamma_i\in \Gamma^\K_{n_i}$, define
\[
  \pi(\gamma_i')_{\varepsilon_\K}:=\pi(\gamma_i)\otimes (\varepsilon_\K\cdot \abs{\,\cdot \,}_\K^{-\nu_i}),
\]
which is a virtual representation of $\GL_{n_i}(\K)$. This is independent of $\varepsilon_\K$ unless $\K=\R$ and $n_i$ is odd. 


\begin{prp} \label{class}
   Let $\gamma'_i:=\gamma_i-\nu_i$ with $\gamma_i\in \Gamma_{n_i}^\K$, and let $\varepsilon_i$ be a character of $\pi_0(\K^\times)$ ($i=1,\ldots, r$). Then 
\[
\times_{i=1}^r \pi(\gamma_i')_{\varepsilon_i} \in \CR^{\rm coge,\K}_{\varepsilon_\K}
\]
if and only if
\[
\begin{cases}
    \textrm{the multiset union }
\gamma:=\gamma_1'\cup \cdots \cup \gamma_r' \in \Gamma_n^\K, \ \textrm{and}&\\
\pi(\gamma_i')_{\varepsilon_i}\cong \pi(\gamma_i')_{\varepsilon_\K}\quad\textrm{for all $i=1,2,\dots,r$.}
\end{cases}
\]
When  this is the case, it holds  that $\times_{i=1}^r \pi(\gamma_i')_{\varepsilon_i} \cong \pi(\gamma)\otimes {\varepsilon_\K}$. 
\end{prp}

\begin{proof}
    This is immediate from Proposition \ref{class00} and the uniqueness of the inducing data (up to reordering) in \eqref{pigamma}.
\end{proof}



\subsection{Zuckerman-Jantzen translations}
Recall the notation from Section \ref{seccoh}. In particular,   $\CR_\varepsilon^{\mathrm{coge},\infty}$ denotes the set of all cohomological irreducible subrepresentations $\pi$ of $\CW_\infty$ such that $\RH^b_\varepsilon(\pi)\neq\{0\}$, and $\pi_\varepsilon$ is the unique representation in $\CR_\varepsilon^{\mathrm{coge},\infty}$ with trivial coefficient system. 

Similar to \eqref{lpsi00}, write $\lambda_\psi: \CW_\infty\rightarrow \C$ for the composition of  
\[  \CW_\infty\xrightarrow{\textrm{projection}}\Ind_{\RN(\rk_\infty)}^{\RG(\rk_\infty)}\psi_{\RN,\infty} \xrightarrow{\textrm{evaluation at the identity}}\C,
\]
and  write $\lambda_\psi: \CF\rightarrow \C$ for the composition of  
\[
  \CF\xrightarrow{\textrm{projection to the $\psi$-component }}\,^{\mathrm{alg}}\Ind_{\RN(\rk\otimes_\BQ \C)}^{\RG(\rk\otimes_\BQ \C)}\C\xrightarrow{\textrm{evaluation at the identity}}\C.
\]

\begin{lemp}\label{iotapi}
   For every $\pi\in \CR_\varepsilon^{\mathrm{coge},\infty}$,  there exists a unique homomorphism  $\iota_\pi\in \Hom_{\RG(\rk_\infty)}(\pi_\varepsilon, F_\pi\otimes \pi)$  such that the diagram 
    \[
\xymatrix{
 \pi_\varepsilon \ar[rr]^{ \iota_\pi}  \ar[rrd]_{\lambda_\psi}&& F_\pi\otimes \pi \ar[d]^{ \lambda_\psi\otimes \lambda_\psi} \\
& &  \C  \\
    }
\]
commutes for all non-trivial characters  $\psi: \rk \backslash \A\rightarrow \C^\times$. Moreover, $\iota_\pi$ induces a linear isomorphism 
\be\label{iotapi2}
\iota_\pi: \RH^b_\varepsilon(\pi_\varepsilon)\rightarrow \RH^b_\varepsilon(\pi).
\ee
\end{lemp}

\begin{proof}
    Note that the equality
    \[
    \mathbf t_a. \lambda_\psi=\lambda_{\psi'}
    \]
    holds in both $\Hom_\C(\CW_\infty,\C)$ and $\Hom_\C(\CF, \C)$. Thus the lemma follows from \cite[Proposition 1.1]{JLS25} (see also \cite[Proposition 2.1]{LLS}). 
\end{proof}


\subsection{The Gauss sum}
Let $\omega':  \A_\mathrm f^\times \rightarrow \C^\times$ be a character. Pick an element $y\in \A_\mathrm f^\times$ such that the character 
\[
\A_\mathrm f\rightarrow \C^\times, \quad x\mapsto \psi_\mathrm f (yx)
\]
has the same conductor as $\omega'$. Let $\widehat \CO$ denote the profinite completion of the ring $\CO$ of integers in $\rk$, which is the ring of integers of $\A_\mathrm f$.  The Gauss sum for $\omega'$ is defined to be
	\begin{equation}\label{gausssum}
		\mathscr G(\omega',\psi):=\int_{\widehat \CO^{\times}}\omega'(yx)^{-1}\cdot\psi_\mathrm f (yx)\mathrm{d}x,
	\end{equation}
	where $\mathrm{d}x$ is the normalized Haar measure so that $\widehat \CO^{\times}$ has volume $1$. We note that the above definition is independent of the choice of $y$.

    Let $\mu_{\infty}\subset \BC^\times$ be the subgroup of the  roots of unity. Recall the cyclotomic character 
\[
{\rm Aut}(\BQ(\mu_{\infty})/\BQ)\to \widehat \BZ^\times, \quad \sigma\mapsto t_{\sigma}
\]
defined by requiring that
\be \label{cyclo}
\sigma( \zeta )=\zeta^{t_{\sigma}} \quad \textrm{for all } \zeta \in \mu_{\infty}.
\ee
Here $\BQ(\mu_{\infty})\subset \C$ is the subfield generated by $\mu_\infty$, and $\widehat \Z$ is the profinite completion of $\Z$.
Write $\sigma\mapsto t_{\sigma, \BA}$ for the composition of 
\[
\Aut(\BC/\BQ)  \xrightarrow{\rm restriction} \Aut(\BQ(\mu_\infty)/\BQ) \xrightarrow{\sigma\mapsto t_{\sigma}}\widehat \BZ^\times \subset \BA_\mathrm f^\times. 
\]
Then for all characters  $\psi_0:\A_\mathrm f\rightarrow  \C^\times$,
\be\label{tsigma}
\sigma(\psi_0(x))=\psi_0(t_{\sigma,\A} x)\quad \textrm{for all $\sigma\in \Aut(\C)$ and $x\in \A_\mathrm f$.} 
\ee
    
 The Gauss sum satisfies the following property which can be easily checked (see \cite[(1.2)]{C24}):
\be \label{gauss}
\sigma({\mathscr G}(\omega',\psi)) = (\sigma\circ\omega')(t_{\sigma,\A})\cdot {\mathscr G}(\sigma\circ \omega',\psi),\quad \textrm{for all }\sigma \in \Aut(\C).
\ee

 When $\omega'$ equals the finite part of a Hecke character $\omega: \rk^\times\backslash \A^\times \rightarrow \C^\times$, we set $\mathscr G(\omega,\psi):=\mathscr G(\omega',\psi)$. 

\subsection{Smooth sesquilinear $\Aut(\C)$-actions}\label{secsmac}

We equip $\Aut(\C)$ with the coarsest topology such that the restriction map
\[
  \Aut(\C)\rightarrow \Aut(\overline \BQ)
\]
is continuous, where $\overline \BQ$ denotes the field of algebraic numbers in $\C$, and the absolute Galois group $\Aut(\overline \BQ)$ is equipped with the usual profinite topology from Galois theory. Then $\Aut(\C)$ is a topological group which is not Hausdorff. 

Let $\mathrm{E}$ be a number field in $\C$. The automorphism group $\Aut(\C/\mathrm{E})$ of $\C$ over $\mathrm{E}$ is an open subgroup of $\Aut(\C)$, and conversely, every open subgroup of $\Aut(\C)$ is uniquely of this  form.

Let $\Gamma$ be a set with an action 
\be\label{actgamma000}
\Aut(\C/\mathrm{E}) \curvearrowright \Gamma, \quad (\sigma, \gamma)\mapsto \,^\sigma \gamma. 
\ee
\begin{dfnp}\label{defsp}
    An element $\gamma\in  \Gamma$ is said to be smooth if  $\Aut(\C/\mathrm{E})_\gamma$ (the stabilizer of $\gamma$) is an open subgroup of $\Aut(\C/\mathrm{E})$. The action \eqref{actgamma000} is said to be locally finite if all elements of $\Gamma$ are smooth. 
\end{dfnp}
In the rest of this subsection we assume that the action \eqref{actgamma000} is locally finite. For each $\gamma\in \Gamma$, write 
\be\label{qgamma}
\mathrm{E}[\gamma]:=\{\textrm{the fixed points in $\C$ of $\Aut(\C/\mathrm{E})_\gamma$}\},
\ee
which is a number field in $\C$ containing $\mathrm{E}$ such that $\Aut(\C/\mathrm{E}[\gamma])=\Aut(\C/\mathrm{E})_\gamma$.

Let $\bigsqcup_{\gamma\in \Gamma} V_\gamma$ be a complex vector bundle with an action
\be\label{actgamma2}
\Aut(\C/\mathrm{E}) \curvearrowright \bigsqcup_{\gamma\in \Gamma} V_\gamma, \quad (\sigma, \phi)\mapsto \,^\sigma \phi
\ee
such that the natural map $\bigsqcup_{\gamma\in \Gamma} V_\gamma\rightarrow \Gamma$ is $\Aut(\C/\mathrm{E})$-equivariant. The action \eqref{actgamma2} is said to be sesquilinear if for all $\sigma\in \Aut(\C/\mathrm{E})$ and $\gamma \in \Gamma$,
\[
\,^\sigma(\phi+\phi')=\,^\sigma\phi+\,^\sigma\phi'\quad\textrm{and}\quad \,^\sigma(a\cdot \phi)=\sigma(a) \cdot \,^\sigma\phi
\]
for all $ \phi, \phi'\in  V_\gamma$ and $a\in \C$. 
When this is the case, the set of smooth elements in $V_\gamma$, to be denoted by $V_\gamma^{\mathrm{sm}}$, is a $\overline \BQ$-subspace of $V_\gamma$.

\begin{dfnp}\label{defsmooth}
    A sesquilinear action \eqref{actgamma2} is said to be smooth if for all $\gamma\in \Gamma$, $V_\gamma^{\mathrm{sm}}$ spans the complex vector space $V_\gamma$.
\end{dfnp}

We recall \cite[Lemma 3.2.1]{Clo} as follows. 

\begin{lemp}\label{ratk}
Let $\RE'$ be a subfield of $\overline \BQ$ and let $V'$ be an $\RE'$-vector space. Let  $V$ be an $\Aut(\C/\RE')$-stable subspace of $\C\otimes_{\RE'} V'$, where  $\Aut(\C/\RE')$ acts on  $\C\otimes_{\RE'} V'$ through its natural action on $\C$. Then 
$V =\C\otimes_{\RE'} (V\cap V')$.   
\end{lemp}

The following result is implied by Lemma \ref{ratk} and \cite[Proposition 11.16]{Spr}. 
\begin{lemp}\label{rationalform0}
  Suppose that  \eqref{actgamma2} is a   sesquilinear action. Then  it is smooth if and only if for all $\gamma\in \Gamma$,  the fixed point set $V_\gamma^{\Aut(\C/\mathrm{E})_\gamma}$ is an  $\mathrm{E}[\gamma]$-form of $V_\gamma$. 
\end{lemp}

Lemma \ref{rationalform0} easily implies the following result.

\begin{lemp}\label{sectionpi00}
    Suppose that  \eqref{actgamma2} is a   smooth sesquilinear action. If 
     $\dim V_\gamma=1$ for every $\gamma\in \Gamma$, then there exists an $\Aut(\C/\mathrm{E})$-equivariant section of the natural map 
    \[
    \bigsqcup_{\gamma\in \Gamma} (V_\gamma\setminus\{0\})\rightarrow \Gamma.
    \]
    Moreover,  
     such a section is unique up to rational multiples in the following sense: for any two such sections $$\varpi, \varpi': \Gamma \rightarrow   \bigsqcup_{\gamma\in \Gamma} (V_\gamma\setminus\{0\}),$$ the map
    \[
    \Gamma \rightarrow \C^\times, \quad \gamma \mapsto \frac{\varpi(\gamma)}{\varpi'(\gamma)} 
    \]
    is $\Aut(\C/\mathrm{E})$-equivariant. Here $\frac{\varpi(\gamma)}{\varpi'(\gamma)} \in \C^\times$ is specified by the equality
    \[
   \varpi(\gamma)= \frac{\varpi(\gamma)}{\varpi'(\gamma)} \cdot \varpi'(\gamma).
    \]
\end{lemp}

When $\Gamma$ is a singleton, the bundle $\bigsqcup_{\gamma\in \Gamma} V_\gamma$ is identified with a complex vector space. Thus the notions and terminologies developed in this subsection also apply  to $\Aut(\C/\mathrm{E})$ actions on complex vector spaces.

\subsection{Some sesquilinear $\Aut(\C)$-actions}\label{secsesq} 

Following \cite[pages 79--80]{H}
    and \cite[page 594]{Mah}, we define 
     \be\label{tnk}
\mathbf{t}_\sigma:=\mathbf{t}_{n, \sigma, \A} := {\rm diag}(t^{n-1}_{\sigma, \A}, \ldots, t_{\sigma, \A}, 1) \in \RG(\A_\mathrm f),
\ee
 and a sesquilinear action  
   \be \label{sigmaf000}
\Aut(\BC)\curvearrowright\Ind^{\RG(\A_\mathrm f)}_{{\rm N}(\A_\mathrm f)} \psi_{\RN,\mathrm f},\quad (\sigma, f)\mapsto \,^\sigma f
\ee
by
\be \label{sigmaf}
 {}^\sigma \! f(g):=\sigma\left( f(\mathbf{t}_\sigma^{-1}  g)  \right), \quad  g\in \RG(\A_\mathrm f).
\ee
It follows from \eqref{tsigma} that the above action is well-defined. One  easily verifies that the diagram 
    \[
\begin{CD}
   \Ind^{\RG(\A_\mathrm f)}_{{\rm N}(\A_\mathrm f)} \psi_{\RN,\mathrm f}@>\sigma >>\Ind^{\RG(\A_\mathrm f)}_{{\rm N}(\A_\mathrm f)} \psi_{\RN,\mathrm f}\\
    @V L_a VV @VV L_a V\\
    \Ind^{\RG(\A_\mathrm f)}_{{\rm N}(\A_\mathrm f)} \psi'_{\RN,\mathrm f}@>\sigma  >>\Ind^{\RG(\A_\mathrm f)}_{{\rm N}(\A_\mathrm f)} \psi'_{\RN,\mathrm f}
\end{CD}
\]
commutes for all  $\sigma\in \Aut(\C)$. Thus  the actions as in      \eqref{sigmaf000} yield a sesquilinear action  
   \be\label{actwhi00}
\Aut(\BC)\curvearrowright\CW_\mathrm f,\quad (\sigma, f)\mapsto \,^\sigma f. 
\ee
Note that the sesquilinear actions \eqref{sigmaf000} and \eqref{actwhi00} are not smooth. 

Similarly, define a sesquilinear action  
   \be \label{sigmaf0}
\Aut(\BC)\curvearrowright \,^{\mathrm{alg}}\Ind_{\RN(\rk\otimes_\BQ \C)}^{\RG(\rk\otimes_\BQ \C)}\C,\quad (\sigma, f)\mapsto \,^\sigma f
\ee
by
\be \label{sigmaF}
 {}^\sigma \! f(g):=\sigma\left( f(\sigma^{-1}.g)  \right), \quad g\in \RG(\rk\otimes_\BQ \C),
\ee
where $\Aut(\BC)$ acts on  $\RG(\rk\otimes_\BQ \C)$ through its natural action on $\C$. Note that this sesquilinear action is smooth. 
It is easily verified that the diagram 
    \[
\begin{CD}
  \,^{\mathrm{alg}}\Ind_{\RN(\rk\otimes_\BQ \C)}^{\RG(\rk\otimes_\BQ \C)}\C@>\sigma >>\,^{\mathrm{alg}}\Ind_{\RN(\rk\otimes_\BQ \C)}^{\RG(\rk\otimes_\BQ \C)}\C\\
    @V L_a VV @VV L_a V\\
    \,^{\mathrm{alg}}\Ind_{\RN(\rk\otimes_\BQ \C)}^{\RG(\rk\otimes_\BQ \C)}\C @>\sigma >>\,^{\mathrm{alg}}\Ind_{\RN(\rk\otimes_\BQ \C)}^{\RG(\rk\otimes_\BQ \C)}\C
\end{CD}
\]
commutes for all  $\sigma\in \Aut(\C)$. Thus the actions as in \eqref{sigmaf0} yield a smooth sesquilinear action  
   \[
\Aut(\BC)\curvearrowright\CF,\quad (\sigma, f)\mapsto \,^\sigma f. 
\]
The above action commutes with the action of $\RG(\rk)\subset\RG(\rk\otimes_{\BQ}\C)$, and hence it induces an action
 \be\label{actbetti000}
\Aut(\BC)\curvearrowright\RH^b(\CX),\quad (\sigma, \phi)\mapsto \,^\sigma \phi. 
\ee
This is also a smooth sesquilinear action.

\section{Franke's filtration}
In this section (and only in this section), we allow $\RG$ to be an arbitrary connected reductive linear algebraic group over $\rk$.  The main purpose of this section is to review Franke's filtration introduced in \cite{F} (see also \cite[Section 3]{G13}, \cite[Section 3]{LS}, and \cite[Section 3]{GG26}), which will be used in the next section.

\subsection{Generalized eigenspace decompositions}

Let $\RP=\RM \ltimes \RN$ be a parabolic subgroup of $\RG$ ($\RM$ is a Levi factor and $\RN$ is the unipotent radical, and similar convention will be used without explanation). Denote by \[
\CA_{\RP}(\RG):=\CA((\RM(\rk)\RN(\A))\backslash \RG(\A))
\]
the space of smooth automorphic forms on $(\RM(\rk)\RN(\A))\backslash \RG(\A)$. We have the usual identification
\[
\CA_{\RP}(\RG)=\Ind_{\RP(\A)}^{\RG(\A)} \CA(\RM)
\]
obtained by applying Frobenius reciprocity to the $(\RM(\rk)\RN(\A))$-invariant functional \[
\begin{array}{rcl}
\Ind_{\RP(\A)}^{\RG(\A)} \CA(\RM)&\rightarrow &\C,\\
\phi&\mapsto& \phi(1)(1),
\end{array}
\]
where $1\in \RM(\A)\subset \RG(\A)$ denotes the identity element. 

Define a complex vector space with a real form 
\[
   \a_{\RP,\C}:=\Hom(\Hom (\RM, \GL_1), \C)\supset \a_{\RP}:=\Hom(\Hom (\RM, \GL_1), \R),
\]
and a homomorphism
\be\label{homoa}
 \RM(\A)\rightarrow \a_{\RP},\quad g\mapsto \left(\chi\mapsto \log(\abs{\chi(g)}_\A)\right). 
\ee
By pullback through the above homomorphism, we identify 
$\C[\a_{\RP,\C}]$ (the space of polynomial functions) as a subspace of $\CA(\RM)$. 

Let $A_\RP$ denote the identity connected component of the real points of the largest central split torus 
in $\mathrm{Res}_{\rk/\BQ}\RM$ ($\mathrm{Res}_{\rk/\BQ}$ indicates the Weil restriction). By using the restriction of \eqref{homoa} to $A_\RP$, $\frak a_{\RP}$ is identified with  $A_\RP$. We also identify $A_\RP$ with its Lie algebra via  the exponential map, and hence $\frak a_{\RP}$ is also identified with the Lie algebra of $A_\RP$. As usual, the differential yields an  identification
\[
\Hom(A_\RP,\C^\times)=\check{\a}_{\RP,\C}
\]
so that every $\lambda\in \check{\a}_{\RP,\C}$ is identified with the character $e^\lambda\in \Hom(A_\RP,\C^\times)$. 
Here and henceforth, `$\check{\ }$' also indicates the dual space of a finite-dimensional vector space. In particular, $\check{\a}_{\RP,\C}$ is the dual space of ${\a}_{\RP,\C}$.

The group $A_\RP$ acts on $\CA(\RM)$ by
\[
  (a.\phi)(x):=\phi(ax), \quad a\in A_\RP, \,\phi\in \CA(\RM),\, x\in \RM(\A).
\]
With respect to this action, we have the generalized eigenspace decomposition
\[
\CA(\RM)=\bigoplus_{\lambda\in \check{\a}_{\RP,\C} } \CA(\RM)_\lambda,
\]
where $\CA(\RM)_\lambda$ is the generalized eigenspace with eigenvalue $\lambda$ (similar notation for generalized eigenspaces will be used without explanation). This induces a decomposition 
\be \label{decompapg1}
\CA_\RP(\RG)=\bigoplus_{\lambda\in \check{\a}_{\RP,\C} } \CA_\RP(\RG)_\lambda,\quad \textrm{where }\ \CA_\RP(\RG)_\lambda:=\Ind_{\RP(\A)}^{\RG(\A)} \CA(\RM)_\lambda.
\ee
For every $\phi\in \CA_\RP(\RG)$ and $\lambda\in \check{\a}_{\RP,\C}$, denote by $\phi_\lambda\in \CA_\RP(\RG)_\lambda$ the component of $\phi$ at $\lambda$ with respect to the decomposition \eqref{decompapg1}. 

\subsection{Cuspidal supports} \label{seccs}

For the convenience of the reader, in this subsection we will review the well-known decomposition of $\CA(\RG)$ in terms of the cuspidal representations of Levi subgroups of $\RG$. See \cite[Section 3]{F}, \cite[Section 1]{FS}, \cite[Chapter 15]{G23} and \cite{GZ} for more details. 

Recall that $\CA_{\mathrm{cusp}}(\RM)\subset \CA(\RM)$ denotes the subspace of smooth cusp forms. Denote by   $\CA_{\mathrm{eis}}(\RM)\subset \CA(\RM)$ the space of all $\phi\in \CA(\RM)$ such that
\[
   \int_{\RM(\rk)\backslash \RM(\A)^\circ} \phi(xg){\phi'(x)} \od\!x =0
\]
for all $\phi'\in \CA_{\mathrm{cusp}}(\RM)$ and $g\in \RM(\A)$. Here $\RM(\A)^\circ$ is the kernel of the homomorphism \eqref{homoa}, and $\od\!x$ is the $\RM(\A)^\circ$-invariant Borel measure of total volume 1.  
Then we have decompositions
\be\label{decompam1}
\CA(\RM)=\CA_{\mathrm{cusp}}(\RM)\oplus \CA_{\mathrm{eis}}(\RM)
\ee
and 
\be\label{decompam2}
\CA_{\mathrm{cusp}}(\RM)=\bigoplus_\Sigma (\C[\a_{\RP,\C}]\otimes \Sigma), 
\ee
where $\Sigma$ runs over all maximal primary  subrepresentations of $\CA_\mathrm{cusp}(\RM)$, and $\C[\a_{\RP,\C}]\otimes \Sigma$ is identified with a subspace of $\CA_\mathrm{cusp}(\RM)$ via the multiplication of functions. Here and as usual,  a representation is said to be primary if it  decomposes into a direct sum of irreducible subrepresentations that are all isomorphic  to each other.
It is known that the space \eqref{decompam2} is nonzero (\cf \cite[Theorem 1.1]{Mu10}), and when $\RM$ is a product of general linear groups, the multiplicity one theorem (see \cite{Sh74}) implies that all maximal primary  subrepresentations of $\CA_\mathrm{cusp}(\RM)$ are irreducible. 

The decompositions \eqref{decompam1} and \eqref{decompam2} induce a decomposition 
\be\label{decomautform}
\CA_{\RP}(\RG)=\left(\bigoplus_\Sigma  \CA_{\RP,\Sigma}(\RG)\right)\oplus \CA_{\RP, \mathrm{eis}}(\RG),
\ee
where 
\be\label{autp}
  \CA_{\RP,\Sigma}(\RG):=\Ind_{\RP(\A)}^{\RG(\A)} (\C[\a_{\RP,\C}]\otimes \Sigma)
  \ee
  and
  \[
  \CA_{\RP, \mathrm{eis}}(\RG):=\Ind_{\RP(\A)}^{\RG(\A)} (\CA_{\mathrm{eis}}(\RM)). 
  \]
 Recall the constant term map
 \[
 \CA(\RG)\rightarrow \CA_{\RP}(\RG), \quad \phi\mapsto \left(\phi_{\RP}: g\mapsto \int_{\RN(\rk)\backslash \RN(\A)} \phi(xg) \od\! x\right),
 \]
 where  $\od\!x$ is the $\RN(\A)$-invariant Borel measure of total volume 1. 
 Write 
  \[
  \CA(\RG)\rightarrow \CA_{\RP,\Sigma}(\RG), \quad \phi\mapsto \phi_{\RP, \Sigma}
  \]
  for the composition of 
  \[
    \CA(\RG)\xrightarrow{(\,\cdot\,)_\RP} \CA_{\RP}(\RG)\xrightarrow{\textrm{projection with respect to \eqref{decomautform}}} \CA_{\RP,\Sigma}(\RG).
  \]
It is known (see \cite[Lemma 3.7]{La2} and \cite[Theorem 4]{HC68}) that for every $\phi\in \CA(\RG)$, 
\be\label{cuspp}
\phi\neq 0\quad \textrm{ if and only if }\quad  \phi_{\RP,\Sigma}\neq 0 \textrm{ for some $\RP$ and $\Sigma$ as above}.
\ee

A cuspidal datum for $\RG$ is defined to be a $\RG(\rk)$-conjugacy class $[\Sigma]:=[(\RM,\Sigma)]$ of a pair $(\RM, \Sigma)$, where $\RM$ is a Levi subgroup of $\RG$ and $\Sigma\subset \CA_{\mathrm{cusp}}(\RM)$ is a maximal primary subrepresentation. 
For each $\phi\in \CA(\RG)$, its cuspidal support $\mathrm{c\text{-}Supp}(\phi)$ is defined to be the set of all cuspidal data $[(\RM, \Sigma)]$ for $\RG$ such that 
$\phi_{\RP,\Sigma}\neq 0$ for some parabolic subgroups $\RP$ of $\RG$ containing $\RM$ as a Levi factor. This definition is independent of the representative $(\RM,\Sigma)$ of the cuspidal datum $[(\RM, \Sigma)]$.

We have the well-known decomposition (see \cite[Theorem 1.4]{FS} and \cite{GZ})
\be\label{decautf}
  \CA(\RG)=\bigoplus_{[\Sigma]} \CA_{[\Sigma]}(\RG),
\ee
where $[\Sigma]$ runs over all cuspidal data for $\RG$, and 
\[
\CA_{[\Sigma]}(\RG):=\left\{\phi\in\CA(\RG)\,:\, \mathrm{c\text{-}Supp}(\phi)\subset \{[\Sigma]\}\right\}.
\]
We call $\CA_{[\Sigma]}(\RG)$ the $[\Sigma]$-component of $\CA(\RG)$. 
For every irreducible smooth automorphic representation $\Pi$ of 
$\RG(\A)$, there is a unique cuspidal datum $[(\RM, \Sigma)]$ for $\RG$ such that $\Pi$ is isomorphic to a subquotient of $\CA_{[\Sigma]}(\RG)$. 
We call this cuspidal datum $[(\RM, \Sigma)]$ the cuspidal support of $\Pi$.   

Applying \eqref{decautf} to $\RM$, we get a  decomposition 
\be\label{decautf2}
  \CA(\RM)=\bigoplus_{[\Sigma_0]_\RM} \CA_{[\Sigma_0]_\RM}(\RM),
\ee
where $[\Sigma_0]_\RM:=[(\RM_0, \Sigma_0)]_\RM$ runs over all cuspidal data for $\RM$. Here $\RM_0$ is a Levi subgroup of $\RM$, $\Sigma_0$ is a maximal primary subrepresentation of $\CA_{\mathrm{cusp}}(\RM_0)$, and $[(\RM_0, \Sigma_0)]_\RM$ is the $\RM(\rk)$-conjugacy class of $(\RM_0, \Sigma_0)$. 
By parabolic induction, \eqref{decautf2} induces a decomposition 
\be\label{decautf3}
  \CA_\RP(\RG)=\bigoplus_{[\Sigma_0]_\RM} \CA_{\RP,[\Sigma_0]_\RM}(\RG),\quad\textrm{where } \CA_{\RP,[\Sigma_0]_\RM}(\RG):=\Ind_{\RP(\A)}^{\RG(\A)}\CA_{[\Sigma_0]_\RM}(\RM). 
\ee

For every cuspidal datum $[\Sigma']$ for $\RG$, write 
\[
 \CA_{\RP,[\Sigma']}(\RG):=\bigoplus_{[\Sigma_0]_\RM\subset [\Sigma']} \CA_{\RP,[\Sigma_0]_\RM}(\RG)\subset \CA_\RP(\RG).
    \]
    Here the direct sum is taken  over all cuspidal data $[\Sigma_0]_\RM$ for $\RM$ such that $[\Sigma_0]_\RM\subset [\Sigma']$ (or equivalently, $[\Sigma_0]=[\Sigma']$). 

The following lemma is an immediate consequence of the definition of cuspidal supports. 
\begin{lemp}
   Let $[\Sigma']$ be a cuspidal datum for $\RG$.  Then for every $\phi\in \CA_{[\Sigma']}(\RG)$, its constant term
    \[
    \phi_\RP\in \CA_{\RP,[\Sigma']}(\RG).
    \]
\end{lemp}

\subsection{Almost square-integrable automorphic forms and Franke's filtration}

Define the space of square-integrable automorphic forms on $\RG(\rk)\backslash \RG(\A)$ by 
\[
\CA^2(\RG):=\{\phi\in \CA(\RG)\,:\, \phi((\,\cdot\,) g)|_{\RG(\rk)\backslash \RG(\A)^\circ}\in \oL^2(\RG(\rk)\backslash \RG(\A)^\circ)\textrm{ for all $g\in \RG(\A)$}\}. 
\]
Here the space $\oL^2(\RG(\rk)\backslash \RG(\A)^\circ)$ is defined with respect to the  $\RG(\A)^\circ $-invariant Borel measure with total volume 1. Similarly, define $\CA^{\bar 2}(\RG)$ to be the space of all $\phi\in \CA(\RG)$ such that 
\[
\phi((\,\cdot\,) g)|_{\RG(\rk)\backslash \RG(\A)^\circ}\in \oL^{2-\epsilon}(\RG(\rk)\backslash \RG(\A)^\circ)\ \textrm{ for all $0<\epsilon<2$ and $g\in \RG(\A)$}. 
\]
Such a function $\phi\in\CA^{\bar 2}(\RG)$ is called an almost square-integrable 
automorphic form on $\RG(\rk)\backslash \RG(\A)$. 

To proceed further, we introduce some more notation. Recall that $\RP=\RM\ltimes \RN$ is a parabolic subgroup of $\RG$. Given a maximal primary subrepresentation $\Sigma\subset \CA_{\mathrm{cusp}}(\RM)$, its exponent, to be denoted by $\mathrm{ep}(\Sigma)$, is the unique element  $\lambda\in \check \a_\RP$ such that $\Sigma\otimes e^{-\lambda}$ is unitarizable. Here $e^{-\lambda}$, which is obviously a character of $\a_\RP$, is viewed as a character of $\RM(\A)$ via the pullback through \eqref{homoa}. We remark   that the exponent defined here is consistent with the one defined in   \eqref{defep} for general linear groups (in the case when $\RP=\RG=\GL_{n'}$ so that $\check \a_\RP=\R$). 

 Define a subset of $\check  \a_\RP$ by 
\[
  ^+\!\check \a_{\RP}:=\begin{cases}
      \left \{\sum_\alpha x_\alpha \alpha\,:\, \textrm{each $x_\alpha$ is a positive real number} \right \} , &\textrm{if } \RP\neq \RG; \\
      \{0\}, &\textrm{if } \RP=\RG,
  \end{cases}
\]
where $\alpha$ runs through all roots of $\RN(\rk_\infty)$ with respect to $A_\RP$. Define $\overline{\,^+\!\check \a_{\RP}}$ to be the closure  of $^+\!\check \a_{\mathsf P}$ in $\check \a_\RP$. 

Suppose that $\RP'=\RM'\ltimes \RN'\supset \RP$ is a parabolic subgroup of $\RG$, with $\RM'\supset \RM$. Then the inclusion  homomorphism $\RM\rightarrow \RM'$ induces a surjective linear map
\be\label{surj5}
 \a_{\RP}\twoheadrightarrow \a_{\RP'}, 
\ee
and the inclusion  homomorphism $A_{\RP'}\rightarrow A_\RP$ induces an injective  linear map
\be\label{inj5}
 \a_{\RP'}\hookrightarrow \a_{\RP}.
\ee
Note that \eqref{inj5} is a section of \eqref{surj5}, and by using these two maps we write 
\[
  \a_{\RP}=\a_{\RP'}\oplus \a_{\RP}^{\RP'}, \quad \textrm{where $\a_{\RP}^{\RP'}$ is the kernel of \eqref{surj5} }. 
\]
By taking the dual spaces, we have that
\[
  \check \a_{\RP}=\check \a_{\RP'}\oplus \check \a_{\RP}^{\RP'}, \quad \textrm{where $\check  \a_{\RP}^{\RP'}$ is the dual space  of $\a_{\RP}^{\RP'}$ }. 
\]
Note that $^+\!\check \a_{\mathsf P}\subset \check \a_\RP^\RG$. 

\begin{lemp}
    Let $\phi\in \CA(\RG)$. Then the followings are equivalent to each other.
    \begin{itemize}
        \item[(a)] $\phi\in \CA^2(\RG)$.
        \item[(b)] For every proper parabolic subgroup $\RP=\RM\ltimes \RN$ of $\RG$ and every $\lambda\in \check \a_{\RP,\C}$
    \[
      (\phi_\RP)_\lambda\neq 0\Rightarrow \Re(\lambda)\in \check \a_\RG - {}^+\!{\check \a}_\RP\qquad(\textrm{$\mathrm{Re}$ indicates the real part}).
    \]
    \item[(c)]
    For every proper parabolic subgroup $\RP=\RM\ltimes \RN$ of $\RG$ and every maximal primary subrepresentation $\Sigma\subset \CA_{\mathrm{cusp}}(\RM)$, 
    \[
      \phi_{\RP,\Sigma}\neq 0\Rightarrow \mathrm{ep}(\Sigma)\in \check \a_\RG- {}^+\!{\check \a}_\RP.
    \]
    \end{itemize}
\end{lemp}
\begin{proof}
     The equivalence of (b) and (c) is easily verified by using the analog of \eqref{cuspp} for $\phi\in \CA_\RP(\RG)$, while the equivalence of (a) and (c) is proved in \cite[Lemma I.4.11]{MW}.
\end{proof}

Similarly, the proof of \cite[Lemma I.4.11]{MW} also shows the following result.
\begin{lemp}
    Let $\phi\in \CA(\RG)$. Then the followings are equivalent to each other. \begin{itemize}
        \item[(a)] $\phi\in \CA^{\bar 2}(\RG)$.
        \item[(b)] For every proper parabolic subgroup $\RP=\RM\ltimes \RN$ of $\RG$ and every $\lambda\in \check \a_{\RP,\C}$
    \[
      (\phi_\RP)_\lambda\neq 0\Rightarrow \Re(\lambda)\in \check \a_\RG-\overline{^+\!{\check \a}_\RP}. 
    \]
    \item[(c)]
    For every proper parabolic subgroup $\RP=\RM\ltimes \RN$ of $\RG$ and every maximal primary subrepresentation $\Sigma\subset \CA_{\mathrm{cusp}}(\RM)$, 
    \[
      \phi_{\RP,\Sigma}\neq 0\Rightarrow \mathrm{ep}(\Sigma)\in \check \a_\RG -\overline{^+\!{\check \a}_\RP}.
    \]
    \end{itemize}
\end{lemp}

Write 
\[
\CA^2_\RP(\RG):=\Ind_{\RP(\A)}^{\RG(\A)} \CA^2(\RM)\subset \CA^{\bar 2}_\RP(\RG):=\Ind_{\RP(\A)}^{\RG(\A)} \CA^{\bar 2}(\RM)\subset \CA_\RP(\RG).
\]
Let   $[\Sigma]$ be a cuspidal datum for $\RG$ as before. For every parabolic subgroup $\RP'$ of $\RG$, write 
\[
  \CA^2_{\RP',[\Sigma]}(\RG):= \CA^2_{\RP'}(\RG)\cap  \CA_{\RP',[\Sigma]}(\RG)\subset  \CA^{\bar 2}_{\RP',[\Sigma]}(\RG):= \CA^{\bar 2}_{\RP'}(\RG)\cap  \CA_{\RP',[\Sigma]}(\RG).
\]
Set 
\[
   \CA^{2,\mathrm{im}}_{\RP',[\Sigma]}(\RG):= \CA^2_{\RP',[\Sigma]}(\RG)\cap \left(\bigoplus_{\lambda\in \check \a_{\RG,\C}+\sqrt{-1}\cdot \check \a_{\RP'}^\RG} \CA_{\RP'}(\RG)_\lambda \right).
\]

Fix a minimal parabolic subgroup $\RP_0=\RM_0\ltimes \RN_0$ of $\RG$. Let
 $\mathscr P$ denote the set of standard parabolic subgroups of $\RG$, namely, parabolic subgroups containing $\RP_0$. We view $\mathscr P$ as a category such that a morphism from  $\RP=\RM\ltimes \RN\in \mathscr P$ to $\RP'=\RM'\ltimes \RN'\in \mathscr P$ ($\RM,\RM'\supset \RM_0$) is an element of 
 \[
   \RM'(\rk)\backslash \{w\in \RG(\rk)\,:\, w \RM w^{-1}=\RM'\}/\RM(\rk), 
 \]
and the identity morphism and the composition of morphisms are obviously defined. Then the assignment $\RP'\mapsto \CA^{2,\mathrm{im}}_{\RP',[\Sigma]}(\RG)$, together with the intertwining operators, forms a functor from $\mathscr P$ to the category of representations of $\RG(\A)$ (see \cite[Section 6]{La2}). The theory of Eisenstein series and intertwining operators then yield a homomorphism (see \cite[Section 5.2]{F} and \cite{BL24, Lap08})
\be\label{eis0}
\varinjlim_{\RP'\in \mathscr P} \CA^{2,\mathrm{im}}_{\RP',[\Sigma]}(\RG)\rightarrow \CA(\RG). 
\ee
The following lemma is a consequence of \cite[Theorem 14]{F} and \cite[Theorem 1.4, Remark 3.4]{FS}.

 \begin{lemp}\label{lembar2}
 The homomorphism \eqref{eis0} induces an isomorphism 
    \[
    \varinjlim_{\RP'\in \mathscr P} \CA^{2,\mathrm{im}}_{\RP',[\Sigma]}(\RG)\cong \CA^{\bar 2}_{[\Sigma]}(\RG).
    \]
    \end{lemp}

 For each $\RP'\in \mathscr P$, denote by $\Aut_{\mathscr P}(\RP')$ the automorphism group of the object $\RP'$ in the category $\mathscr P$.  
As a consequence of Lemma \ref{lembar2}, we have that 
    \be\label{dec2bar}
    \bigoplus_{[\RP']\subset \mathscr P} \CA^{2,\mathrm{im}}_{\RP',[\Sigma]}(\RG)_{\Aut_{\mathscr P}(\RP')}\cong \CA^{\bar 2}_{[\Sigma]}(\RG),
    \ee
    where $[\RP']$ denotes the isomorphism  class of $\RP'$ in the category $\mathscr P$, and the subscript group `$\Aut_{\mathscr P}(\RP')$' indicates the coinvariant space.



By using the Killing form of the  Lie algebra of $\RG(\rk_\infty)$, the space $\a_\RP^\RG$ is naturally an inner product space, and hence its dual space $\check \a_\RP^\RG$ is also an inner product space. The inner product of the latter space is denoted by $\la\,\cdot,\cdot\,\ra_{\check \a_\RP^\RG}$.
Set
\[
\check \a_\RP^+:=\check \a_\RG+\{x\in \check \a_\RP^\RG\,:\, \la x,y\ra_{\check \a_\RP^\RG}>0 \textrm{ for all }y\in \overline{^+\!\check \a_\RP}\setminus \{0\}\}. 
\]
Write $\rho_0^\vee\in \a_{\RP_0}^\RG$ for the element such that
\[
\la \rho_0^\vee, \alpha\ra=1\ \ \textrm{ for all $\alpha\in \check \a_{\RP_0}^\RG$ that is a simple root with respect to $\RN_0(\rk_\infty)$}.
\]

We formulate Franke's filtration as in the following proposition. It is essentially proved in \cite[Section 6]{F} (see also \cite[Theorem 3.3]{LS}, and 
a similar result in \cite[Section 3]{G13}).

\begin{prp}\label{franke}
    There exists a family 
   \[ \{  \CA_{[\Sigma]}(\RG)_{[\leq t]}\}_{t\in \R} 
   \]
   of subrepresentations of $\CA_{[\Sigma]}(\RG)$ with the following properties: 
   \begin{itemize}
       \item[(a)] $\CA_{[\Sigma]}(\RG)_{[\leq t]}=0$ whenever $t<0$;
       \item[(b)] $\CA_{[\Sigma]}(\RG)_{[\leq t_1]}\subset \CA_{[\Sigma]}(\RG)_{[\leq t_2]}$ whenever $t_1< t_2$;
       \item[(c)] $\bigcup_{t\in \R}\CA_{[\Sigma]}(\RG)_{[\leq t]}=\CA_{[\Sigma]}(\RG)$; 
       \item[(d)] the set $\{  \CA_{[\Sigma]}(\RG)_{[\leq t]}\,:\,t\in \R\}$ is finite;  
       \item[(e)] 
       for every $t\in \R$, 
        the map \[
        \begin{array}{rcl}
      \CA_{[\Sigma]}(\RG)_{[\leq t]}&\rightarrow & \bigoplus_{\RP'\in \mathscr P} \CA_{\RP'}(\RG),\smallskip \\
        \phi&\mapsto & \sum_{\RP'\in \mathscr P}  \sum_{\lambda\in \check \a_{\RP',\C}, \,\Re(\lambda)\in \check \a_{\RP'}^+,\,\la \Re(\lambda), \rho_0^\vee\ra=t}(\phi_{\RP'})_\lambda
   \end{array}
\]
induces an isomorphism 
\be\label{isogr}
  \CA_{[\Sigma]}(\RG)_{[t]}\cong  \bigoplus_{\RP'\in \mathscr P} \, \bigoplus_{\lambda\in \check \a_{\RP',\C}, \,\Re(\lambda)\in \check \a_{\RP'}^+,\,\la \Re(\lambda), \rho_0^\vee\ra=t}\CA_{\RP',[\Sigma]}^{\bar 2}(\RG)_\lambda
\ee
of  representations of $\RG(\A)$, where
\[
\CA_{[\Sigma]}(\RG)_{[t]}:=\frac{\CA_{[\Sigma]}(\RG)_{[\leq t]}}{\CA_{[\Sigma]}(\RG)_{[ <t]}}, \quad \textrm{and }\  \CA_{[\Sigma]}(\RG)_{[< t]}:=\bigcup_{t'\in \R, t'<t} \CA_{[\Sigma]}(\RG)_{[\leq t']}. 
\]
  \end{itemize}
\end{prp}


\section{Cohomological tamely isobaric  representations}

We retain the notation of the last section, and assume as in the Introduction that $\RG=\GL_n$ ($n\geq 1$). Further assume that $\RP_0=\RM_0\ltimes \RN_0$ is the Borel subgroup of upper-triangular matrices, with $\RM_0$ the diagonal torus. 

        \subsection{Isobaric sums and standard modules}\label{secisob}
    We fix a partition 
    \[
     n=n_1+n_2+\dots+n_r \qquad ( r, n_1,n_2,\dots\, n_r\geq 1)
    \]
    and write 
    \[
     \RM=\GL_{n_1}\times \GL_{n_2}\times \dots \times \GL_{n_r}, 
     \]
    which is viewed as a Levi subgroup of $\RG$ as usual. 
Let  $\RP=\RM\ltimes \RN\in\mathscr P$ be the standard parabolic subgroup of $\RG$ that contains $\RM$ as a Levi factor.

    Let $v$ be a local place of $\rk$, and denote by  $\abs{\cdot}_v$ the normalized absolute value on the local field $\rk_v$. Suppose that $\pi_i$ ($i=1,2,\dots, r$) is an irreducible admissible smooth representation of $\GL_{n_i}(\rk_v)$ (in the archimedean case, by convention this means an irreducible Casselman-Wallach representation). The isobaric sum 
    \[
    \pi_1\boxplus \pi_2\boxplus \dots\boxplus \pi_r
    \]
    is the  irreducible admissible smooth representation of $\RG(\rk_v)$ whose Langlands parameter is the direct sum of those of $\pi_i$'s. See \cite{Kn91, Ku91, HT01} for the notion of Langlands parameters, which are completely reducible finite-dimensional  representations of the Weil-Deligne group. Recall that the above isobaric sum is generic if and only if all $\pi_i$'s are generic and the representation $\Ind_{\RP(\rk_v)}^{\RG(\rk_v)}(\pi_1\widehat \otimes  \pi_2\widehat \otimes  \dots\widehat \otimes \pi_r)$ is irreducible (see \cite{BZ77, JS83, Vo78, Z80}). 

    Recall that by taking the Langlands quotients (the unique irreducible quotients), the isomorphism classes of   standard modules of $\RG(\rk_v)$  are in bijection with the isomorphism classes of  irreducible admissible smooth representations of $\RG(\rk_v)$ (see  \cite{Lan89, Z80}).  
    
    \begin{thmp}\label{prop:std}
    Let  $\pi_i$ be a unitarizable generic irreducible admissible smooth representation of $\GL_{n_i}(\rk_v)$, and let $\nu_i\in \mathbb R$ ($i=1,2, \dots, r)$. Suppose that $\nu_1\geq \nu_2\geq \dots \geq \nu_r$. Then 
    \[
    \Ind_{\RP(\rk_v)}^{\RG(\rk_v)} \left((\pi_1\otimes \abs{\det}^{\nu_1}_{v})\,\widehat \otimes\, (\pi_2\otimes \abs{\det}^{\nu_2}_{v})\,\widehat \otimes\, \dots \,\widehat \otimes\, (\pi_r\otimes \abs{\det}^{\nu_r}_{v})\right)
    \]
    is isomorphic to the standard module whose Langlands quotient is isomorphic to the isobaric sum
    \[
    (\pi_1\otimes \abs{\det}^{\nu_1}_{v})\boxplus (\pi_2\otimes \abs{\det}^{\nu_2}_{v})\boxplus \dots \boxplus(\pi_r\otimes \abs{\det}^{\nu_r}_{v}).
    \] 
\end{thmp}

\begin{proof}
    For $i=1,2,\ldots, r$, write $\pi_i\otimes\abs{\det}^{\nu_i}_v$ as a normalized induction 
    \[
    \pi_i\otimes\abs{\det}^{\nu_i}_v  = \delta_{l_{i-1}+1}\times \delta_{l_{i-1}+2}\times \cdots \times \delta_{l_i}
    \]
    from a block upper-triangular parabolic subgroup of $\GL_{n_i}(\rk_v)$, where $l_0:=0 < l_1 <\cdots < l_r$ and $\delta_j$, $j=1,2,\ldots, l_r$, are essentially square-integrable 
    representations of general linear groups over $\rk_v$. Similar notation for normalized inductions from block upper-triangular parabolic subgroups will be used below.
     
    For $j=1,2,\ldots, l_r$, write ${\rm ep}(\delta_j)$ for the real number such that $\delta_j \otimes \abs{\det}_v^{-{\rm ep}(\delta_j)}$ is unitarizable. 
    By the classification of unitarizable generic representations of general linear groups (see \cite{T,V} and \cite[Theorem 8.2 and p.1155]{BR}), we have 
    \[
    | {\rm ep}(\delta_j) - \nu_i| <\frac{1}{2}\quad \text{for all }l_{i-1}<j \leq l_i.
    \]
  Assume that $1\leq j<k \leq l_r$. Then 
    \[
    {\rm ep}(\delta_k) < {\rm ep}(\delta_j)+1.
    \]
    It follows that
    \begin{itemize}
     \item if $v$ is non-Archimedean and the segments of supercuspidal representations associated to $\delta_j$ and $\delta_k$ are linked in the sense of \cite[Section 4.1]{Z80}, then 
    ${\rm ep}(\delta_j)\geq {\rm ep}(\delta_k)+1$;
    \item if $v$ is Archimedean and $\delta_j$ and $\delta_k$ are linked in the sense of \cite[I.7]{MW89}, then 
    ${\rm ep}(\delta_j)\geq {\rm ep}(\delta_k)+1$.
    \end{itemize}
    
    Denote the induced representation in the theorem by $\pi$. Suppose that ${\rm ep}(\delta_k) = \max_{j=1,\ldots, l_r}{\rm ep}(\delta_j)$. By \cite[Theorem 4.2]{Z80} and \cite[Proposition I.9]{MW89}, $\delta_j \times \delta_k$ is irreducible for all $1\leq j<k$, which implies that 
    \[
    \pi \cong \delta_k \times \delta_1 \times\delta_2\times \cdots \times \delta_{k-1} \times \delta_{k+1} \times\cdots \times \delta_{l_r}.
    \]
    Repeating this process, we see that $\pi$ is isomorphic to a standard module with the  desired Langlands quotient.
\end{proof}

\subsection{Eisenstein series} 
As in the Introduction, let  
\be\label{isopi}
\Pi:=\Sigma_{1}\boxplus  \Sigma_{2}\boxplus  \dots \boxplus \Sigma_{r} 
\ee
be an isobaric automorphic representation of $\GL_n(\A)$,
where  $\Sigma_i$ ($1\leq i\leq r$) is an irreducible cuspidal smooth automorphic representation of $\GL_{n_i}(\A)$ and the exponents satisfy 
\[
		\mathrm{ep}(\Sigma_1)\geq\mathrm{ep}(\Sigma_2)\geq\dots\geq\mathrm{ep}(\Sigma_r).
		\]
Note that the cuspidal support of $\Pi$ is $[(\RM, \Sigma)]$,  where 
 \[
         \Sigma=\Sigma_1\,\widehat{\otimes} \,\Sigma_2\,\widehat{\otimes} \,\cdots \,\widehat{\otimes} \,\Sigma_r\subset\CA_{\mathrm{cusp}} (\RM).
         \]
      
    For every $\phi\in \CA_{\RP,\Sigma}(\RG)$ and $\underline{s}:=(s_1,\dots,s_r)\in\C^r$, define a function $\phi_{\underline{s}}$ on $\RG(\A)$ such that
    \[
    \phi_{\underline{s}}(g)=\phi(g) \cdot \prod_{i=1}^r\abs{\det m_i}_{\A}^{s_i}
    \]
    for all $g=\diag(m_1, m_2, \dots, m_r) u k \in\RG(\A)$ with $m_i\in \GL_{n_i}(\A)$ ($i=1,2,\dots, r$),  $u\in \RN(\A)$, and $k\in K_\infty \GL_n(\widehat \CO)$.  Consider the 
    Eisenstein series 
    \be\label{defeis}
\RE(\underline{s},\phi;g):=\sum_{\gamma\in\RP(\rk)\backslash\RG(\rk)}\phi_{\underline{s}}(\gamma g),\qquad g\in\RG(\A).
    \ee
   It converges absolutely when $\mathrm{Re}(s_i-s_j)\gg0$ for all $1\leq i<j\leq r$, and admits a meromorphic continuation to $\C^r$. Moreover, if $\RE(\underline{s},\phi;g)$ is holomorphic at $\underline{s}=\underline 0:=(0,0,\ldots, 0)\in \BC^r$ for all $g\in \RG(\A)$, then
   \be\label{ce00}
   \CE(\phi):=\RE(\underline{s}, \phi; \,\cdot\,)|_{\underline{s}=\underline{0}}\in \CA_{[\Sigma]}(\RG). 
   \ee
See \cite[Section 5.2]{F}, \cite{Lap08}, and \cite[Theorem 2.3]{BL24} for the general theory of Eisenstein series. The assertion that $\CE(\phi)\in \CA_{[\Sigma]}(\RG)$
follows from \cite[Theorem 1.4, Remark 3.4]{FS}.

   Recall from the Introduction the induced representation 
	\begin{equation}
		\label{Sigma}
\mathrm I(\Pi):=\mathrm{Ind}^{\RG(\A)}_{\RP(\A)}\Sigma=\widehat\otimes_v'\RI_v(\Pi)\subset \CA_{\RP,\Sigma}(\RG),
	\end{equation}
where 
\[
\RI_v(\Pi):=\mathrm{Ind}^{\RG(\rk_v)}_{\RP(\rk_v)}\Sigma_v.
\]
By Theorem \ref{prop:std}, $\RI_v(\Pi)$ is isomorphic to the standard module whose Langlands quotient is isomorphic to $\Pi_v$. 

\begin{prp}\label{lemti0}
Assume that $\Pi$ is tamely isobaric (see Definition \ref{def:isobaric}). 
Then the following assertions hold true.

\noindent (a)  The  Eisenstein series   \eqref{defeis} is holomorphic at $\underline s=\underline{0}$ 
  for all $\phi\in \CA_{\RP,\Sigma}(\RG) $ and $g\in \RG(\A)$.

   \noindent (b) The  diagram
    \be\label{iso2}
\xymatrix{
\CA_{\RP,\Sigma}(\RG)\ar[rr]^{ \CE}  \ar[rrd]_{=}&& \CA_{[\Sigma]}(\RG) \ar[d]^{ (\,\cdot\,)_{\RP,\Sigma}} \\
& &  \CA_{\RP,\Sigma}(\RG)  
    }
\ee
commutes.   

   
    \noindent (c) The Whittaker integral  $\lambda_\psi\circ \CE|_{\RI(\Pi)}\neq 0$. 
  
   \end{prp}

\begin{proof} (a) 
For each $w\in \frak S_r$ (the symmetric group), write 
\[
w\RM:=\GL_{n_{w^{-1}(1)}}\times \GL_{n_{w^{-1}(2)}}\times \dots \times \GL_{n_{w^{-1}(r)}},
\]
and
\[
w(\Sigma[\underline{s}]):=\Sigma_{w^{-1}(1)}\abs{\det}_\A^{s_{w^{-1}(1)}}\widehat{\otimes} \Sigma_{w^{-1}(2)}\abs{\det}_\A^{s_{w^{-1}(2)}}\widehat{\otimes} \dots\widehat{\otimes} \Sigma_{w^{-1}(r)}\abs{\det}_\A^{s_{w^{-1}(r)}},
\]
where $\underline s=(s_1, s_2, \dots, s_r)\in \C^r$, and 
$\Sigma[\underline{s}]:=\Sigma_1\abs{\det}_\A^{s_1}\widehat{\otimes} \Sigma_2\abs{\det}_\A^{s_2}\widehat{\otimes} \dots\widehat{\otimes} \Sigma_r\abs{\det}_\A^{s_r}$. 
Following the notation in \cite{MW89}, we have the intertwining operator 
\be\label{intertwin}
M(w, \Sigma, \underline{s}) = r(w,\Sigma,\underline{s}) N(w, \Sigma, \underline{s}): \mathrm{Ind}^{\RG(\A)}_{\RP(\A)}(\Sigma[\underline{s}]) \to \mathrm{Ind}^{\RG(\A)}_{\RP'(\A)}(w(\Sigma[\underline{s}])),
\ee
where $\RP'$ is the standard parabolic subgroup of $\RG$ with Levi factor $w\RM$, 
\[
r(w,\Sigma,\underline{s}) := \prod_{1\leq i<j \leq r, \, w(i)>w(j)} \frac{\oL(s_i-s_j, \Sigma_i\times \Sigma_j^\vee)}{\oL(s_i-s_j+1, \Sigma_i\times \Sigma_j^\vee)\varepsilon(s_i-s_j, \Sigma_i\times \Sigma_j^\vee)},
\]
and $N(w, \Sigma, \underline{s})$ is the normalized intertwining operator. 

By \cite[Theorem 5.2]{S81}, \cite[Theorem 5.3]{JS81a}, \cite[Proposition 3.6]{JS81b} and the second tame condition, $r(w, \Sigma, \underline{s})$ is holomorphic at $\underline{s}=\underline{0}$.
By \cite[Proposition I.10]{MW89}, $N(w,\Sigma, \underline{s})$ is holomorphic at $\underline{s}=\underline{0}$ as well.
It follows that the intertwining operator $M(w,\Sigma, \underline{s})$ in \eqref{intertwin} is holomorphic at $\underline{s}=\underline{0}$. Then by the constant term formula for Eisenstein series (\cite[II.1.7]{MW}) and \cite[Lemma 6.2]{La2}, we see that 
\be\label{holo1}
\textrm{$\RE(\underline{s},\phi;g)$ is holomorphic at $\underline{s}=\underline 0\in \BC^r$ for all $\phi\in \RI(\Pi)$ and $g\in \RG(\A)$. }
\ee


It is easy to see that $\CA_{\RP,\Sigma}(\RG)$ is spanned by functions of the form  $\varphi = D \phi_{\underline{\nu}}|_{\underline{\nu}=\underline{0}}$, where
$\phi\in \RI(\Pi)$ and $D$ is a constant coefficient differential operator on the complex vector space $\check \a_{\RP,\C} =\C^r$. If $\mathrm{Re}(s_i-s_j)\gg0$ for all $1\leq i<j\leq r$, then 
\[
\begin{aligned}
\RE(\underline{s},\varphi;g) & =\sum_{\gamma\in\RP(\rk)\backslash\RG(\rk)}\varphi_{\underline{s}}(\gamma g) \\
& =\sum_{\gamma\in\RP(\rk)\backslash\RG(\rk)}{((D \phi_{\underline{\nu}})}|_{\underline{\nu}=\underline{0}})_{\underline s}(\gamma g)\\
& =D\left.\left(\sum_{\gamma\in\RP(\rk)\backslash\RG(\rk)}\phi_{\underline{\nu}+\underline{s}}(\gamma g)\right)\right|_{\underline \nu=0}\\
&=D\left(\sum_{\gamma\in\RP(\rk)\backslash\RG(\rk)}\phi_{\underline{s}}(\gamma g)\right),
\end{aligned}
    \]
    where we have used the Lebesgue Dominated Convergence Theorem (\cite[Theorem 1.34]{Ru87}) and Morera's Theorem (\cite[Theorem 10.17]{Ru87}) in the second last equality.
By \eqref{holo1}, this implies part (a) of the proposition. 

\noindent (b) 
Take $\varphi = D\phi_{\underline{\nu}}|_{\underline{\nu}=\underline{0}}\in \CA_{\RP,\Sigma}(\RG)$ as in the proof of (a). By the constant term formula
and the above argument,  we have that 
\[
\CE(\varphi)_\RP =\varphi+  \sum_{w\in \frak S_r\setminus\{{\rm id}\}, \, w\RM =\RM}D\left(M(w, \Sigma, \underline{\nu})\phi_{\underline{\nu}}\right)|_{\underline{\nu}=\underline{0}},
\]
where `$\mathrm{id}$' denotes the identity element of the symmetric group $\frak S_r$. 
Then the first tame condition implies that $\CE(\varphi)_{\RP,\Sigma}=\varphi$. This proves part (b) of the proposition. 


\noindent (c) Let $S$ be a finite set of places of $\rk$ including all the Archimedean places such that  for all local places $v\notin S$ of $\rk$, 
\begin{itemize}
    \item  
$\Sigma_v$  is unramified; and 
\item the conductor of the $v$-component $\psi_v$ of $\psi$ equals the ring of integers $\CO_v$ of $\rk_v$. 
\end{itemize}

By the Fourier coefficient formula (see \cite[Chapter 7]{S10}), 
there exist Whittaker functionals 
\[
0\neq \lambda_v \in \Hom_{\RN(\rk_v)}(\RI_v(\Pi), \psi_{\RN, v})\qquad (v\in S)
\]
with the following property:  
if $\phi = \otimes_v \phi_v \in \RI(\Pi)$, with $\phi_v \in \RI_v(\Pi)$  the spherical vector whenever $v\not\in S$, then
\[
\lambda_\psi(\CE(\phi)) = \frac{1}{\prod_{1\leq i<j\leq r}\oL^S(1, \Sigma_i\times \Sigma_j^\vee)} \cdot \prod_{v\in S}\lambda_v(\phi_v).
\]
Here $\oL^S(s, \Sigma_i\times \Sigma_j^\vee)$ denotes the partial L-function outside $S$.  For all $1\leq i<j\leq r$, by \cite[Proposition 3.6]{JS81b}, the first tame condition implies that $\oL^S(s, \Sigma_i\times \Sigma_j^\vee)$ is  holomorphic at $s=1$; and  it is  nonzero at $s=1$ by  \cite[Theorem 5.2]{S81} and \cite[Theorem 5.3]{JS81a}.  Hence there exists $\phi \in \RI(\Pi)$ such that $\lambda_\psi(\CE(\phi))\neq 0$.
\end{proof}

  Suppose that $\Pi$ is tamely isobaric. 
Denote by $\RI^{\mathrm{aut}}(\Pi)$ the image of $\RI(\Pi)$ under the map $\CE$. 

\begin{remarkp}\label{iaut}
    Let $(\RM',\Sigma',\RP')$ be a triple such that $(\RM',\Sigma')\in [(\RM,\Sigma)]$ and $\RP'$ is a parabolic subgroup of $\RG$ containing $\RM'$ as a Levi factor. Attached to this triple we form the Eisenstein series $\RE(\nu, \phi;g)$ as in \eqref{defeis}, where $\phi\in \Ind_{\RP'(\A)}^{\RG(\A)} \Sigma'$, $g\in \RG(\A)$ and $\nu\in \check \a_{\RP',\C}$. Proposition \ref{lemti0} implies that these Eisenstein series are holomorphic at $\nu=0$, and it  follows from 
    \cite[Theorem IV.1.10]{MW} that 
    \[
      \{\RE(\nu,\phi;\,\cdot\,)|_{\nu=0}\,:\, \phi \in \Ind_{\RP'(\A)}^{\RG(\A)} \Sigma'\}=\RI^{\mathrm{aut}}(\Pi). 
    \]
    In particular, the subrepresentation $\RI^{\mathrm{aut}}(\Pi)\subset\CA(\RG)$ is independent of the representative $(\RM,\Sigma)$ of the cuspidal datum. 
\end{remarkp}

Recall from the Introduction the decomposition $
\mathrm I(\Pi)=\mathrm I_\infty(\Pi)\otimes \mathrm I_\mathrm f(\Pi).
$

   \begin{lemp}\label{lem20}
      One has that
       \be\label{twohom}
     \dim \Hom_{\RG(\rk_\infty)}(\RI_\infty(\Pi), \CW_\infty)=  \dim \Hom_{\RG(\A_\mathrm f)}(\RI_\mathrm f(\Pi), \CW_\mathrm f)= 1.
       \ee
       Moreover,  there are elements  \[
    0\neq  \jmath_{\Pi,\infty}\in \Hom_{\RG(\rk_\infty)}(\RI_\infty(\Pi), \CW_\infty)\quad \textrm{and}\quad    0\neq  \jmath_{\Pi,\mathrm f}\in  \Hom_{\RG(\A_\mathrm f)}(\RI_\mathrm f(\Pi), \CW_\mathrm f)
       \]
 such that the diagram 
 \be\label{comwhii}
\begin{CD}
   \RI(\Pi)=\RI_\infty(\Pi)\otimes \RI_{\mathrm f}(\Pi)@>\CE \ (\textrm{see \eqref{ce00}}) >>\CA(\RG)\\
    @V\jmath_{\Pi,\infty}\otimes \jmath_{\Pi, \mathrm f} VV @VV\jmath 
  V\\
    \CW_\infty\otimes \CW_{\mathrm f}@>\subset   >>\CW
\end{CD}
\ee
commutes. 
   \end{lemp}

   \begin{proof}
     Part (c) of Proposition \ref{lemti0} implies that the two spaces of homomorphisms in \eqref{twohom}  are at least one-dimensional.  By \cite{Ro73} and 
     \cite[Theorem 15.6.7]{W92}, they are at most one-dimensional. 

    By Frobenius reciprocity, the second assertion of the lemma follows from the factorization of the Whittaker periods of Eisenstein series (see \cite[Proposition 7.1.2]{S10}). 
   \end{proof}

Let  $\jmath_{\Pi,\infty}$ and $\jmath_{\Pi,\mathrm f}$ be as in the above lemma. Let 
$\RI^{\mathrm{whi}}(\Pi)\subset \CW$ denote the image of $\RI^{\mathrm{aut}}(\Pi)$ under the homomorphism $\jmath$ in \eqref{jmath}, which is nonzero by part (c) of Proposition \ref{lemti0}. 
Denote by $\RI_\infty^{\mathrm{whi}}(\Pi)\subset \CW_\infty$ the image of $\jmath_{\Pi,\infty}$, and likewise denote by $\RI_\mathrm f^{\mathrm{whi}}(\Pi)\subset \CW_\mathrm f$ the image of $\jmath_{\Pi,\mathrm f}$. Then 
\[
\RI^{\mathrm{whi}}(\Pi)=\RI_\infty^{\mathrm{whi}}(\Pi)\otimes \RI_\mathrm f^{\mathrm{whi}}(\Pi).
\]
Recall that $\RI_v(\Pi)$ is isomorphic to a standard module for every place $v$ of $\rk$ by Theorem \ref{prop:std}. Hence by \cite[Lemma 2.5]{J09} and \cite{JS83}, we have an isomorphism 
\[
\jmath_{\Pi,\infty}\otimes \jmath_{\Pi, \mathrm f}: \RI(\Pi) \xrightarrow{\sim} \RI^{\mathrm{whi}}(\Pi).
\]





\subsection{$[\Sigma]$-component in the tamely isobaric case}

Assume that $\Pi$ is tamely isobaric. This subsection is devoted to a proof of the following theorem, which generalizes some results in \cite[Section 1.5]{G18}
and \cite[Section 3.1]{C26}.

\begin{thmp}\label{lemti00}
 The two arrows in 
  \[
 \CA_{\RP,\Sigma}(\RG)\xrightarrow{\CE}  \CA_{[\Sigma]}(\RG)\xrightarrow{(\,\cdot\,)_{\RP,\Sigma}} \CA_{\RP,\Sigma}(\RG)
  \]
  are both  isomorphisms of representations of $\RG(\A)$. 
\end{thmp}

Write 
\[
  r= d_1+d_2+\dots+d_{r'}
\]
for the partition such that
\begin{eqnarray*}
&&\quad\,   \mathrm{ep}(\Sigma_1)=\mathrm{ep}(\Sigma_2)=\dots = \mathrm{ep}(\Sigma_{d_1})\\
&&> \mathrm{ep}(\Sigma_{d_1+1})=\mathrm{ep}(\Sigma_{d_1+2})=\dots = \mathrm{ep}(\Sigma_{d_1+d_2})\\
&&>\dots \, \dots \\
&&>\mathrm{ep}(\Sigma_{d_1+d_2+\dots + d_{r'-1}+1})= \mathrm{ep}(\Sigma_{d_1+d_2+\dots + d_{r'-1}+2})=\dots =  \mathrm{ep}(\Sigma_{d_1+d_2+\dots +d_{r'}}). 
\end{eqnarray*}
This gives a corresponding partition 
\[
n = m_1+m_2+\cdots+m_{r'},
\]
where for each $j=1,2,\ldots, r'$ we put
\[
m_j:=n_{d_1+d_2+\cdots+d_{j-1}+1}+n_{d_1+d_2+\cdots+d_{j-1}+2}+\cdots+n_{d_1+d_2+\cdots+d_j}.
\]
Denote by $\RP_\Sigma=\RM_\Sigma\ltimes \RN_\Sigma\in \mathscr P$ the standard parabolic subgroup associated to this partition of $n$, with $\RM_\Sigma=\GL_{m_1}\times \GL_{m_2}\times \dots\times \GL_{m_{r'}}$.

\begin{lemp}\label{m2m1}
    Let $\RM_1$ be a Levi subgroup of $\RG$, and let $[\Sigma_1]_{\RM_1}:=[(\RM_2, \Sigma_1)]_{\RM_1}$ be a cuspidal datum for $\RM_1$. Assume that 
    \[
    [\Sigma_1]_{\RM_1}\subset [\Sigma]\quad \textrm{and}\quad \CA^2_{[\Sigma_1]_{\RM_1}}(\RM_1)\neq \{0\}.
    \]
    Then $\RM_2=\RM_1$. 
\end{lemp}
\begin{proof}
In view of the second tame condition, this follows from the description of the discrete spectrum for general linear groups  in \cite{MW89}. 
\end{proof}

Write $\lambda_\Sigma\in \check \a_{\RP_\Sigma,\C}$ for the element such that $A_{\RP_\Sigma}$ acts on $\Ind_{\RM(\A)}^{\RM_\Sigma(\A)}\Sigma$ by the character $e^{\lambda_\Sigma}$. 

\begin{lemp}\label{lem25}
    Let $\RP'=\RM'\ltimes \RN'\in \mathscr P$ ($\RM'\supset \RM_0$) and let $\lambda\in \check \a_{\RP',\C}$. If 
    \be\label{relambda}
    \Re(\lambda)\in \check \a_{\RP'}^+\quad \textrm{and}\quad \CA_{\RP',[\Sigma]}^{\bar 2}(\RG)_\lambda\neq \{0\},
    \ee
    then $\RP'=\RP_\Sigma$ and $\lambda=\lambda_\Sigma$. 
\end{lemp}
\begin{proof}
    The assumption of the lemma implies that
   \be\label{notzero}
   \CA_{[\Sigma']_{\RM'}}^{\bar 2}(\RM')_\lambda\neq \{0\} 
   \ee
   for some cuspidal datum $[\Sigma']_{\RM'}$ for $\RM'$ such that $[\Sigma']_{\RM'}\subset [\Sigma]$. Applying Lemma \ref{lembar2} to $\RM'$, we get a standard  parabolic subgroup $\RP''=\RM''\ltimes \RN''\subset \RP'$ of $\RG$ (with $\RM_0\subset \RM''\subset \RM'$) such that
     \be\label{notzero2}
   \CA_{[\Sigma'']_{\RM''}}^{ 2}(\RM'')_{\lambda'}\neq \{0\}, 
   \ee
   where $[\Sigma'']_{\RM''}$ is a cuspidal datum for $\RM''$ such that $[\Sigma'']_{\RM''}\subset [\Sigma']_{\RM'}$, and $\lambda'\in \check \a_{\RP'',\C}$ is an element such that 
   \be\label{lamdbdap}
     \lambda'|_{\a_{\RP',\C}}=\lambda\quad \textrm{and}\quad  \mathrm{Re}(\lambda')|_{\a_{\RP''}^{\RP'}}=0.
   \ee
The equalities in \eqref{lamdbdap} and the first condition in \eqref{relambda} imply that
\be\label{rel2}
\mathrm{Re}(\lambda')\in \check \a_{\RP'}^+\subset \check \a_{\RP''}.
\ee

In view of \eqref{notzero2}, 
 Lemma \ref{m2m1} implies that  there is a permutation $w\in \frak S_r$ (the symmetric group) such that
\[
\RM''=\GL_{n_{w^{-1}(1)}}\times \GL_{n_{w^{-1}(2)}}\times \dots \times \GL_{n_{w^{-1}(r)}}
\]
and
\[
\Sigma''=\Sigma_{{w^{-1}(1)}}\widehat{\otimes} \Sigma_{{w^{-1}(2)}}\widehat{\otimes} \dots\widehat{\otimes} \Sigma_{{w^{-1}(r)}}\subset\CA_{\mathrm{cusp}} (\RM'').
\]
Then \eqref{notzero2} implies that
\[
\Re(\lambda')=(\mathrm{ep}(\Sigma_{{w^{-1}(1)}}), \mathrm{ep}(\Sigma_{{w^{-1}(2)}}), \dots, \mathrm{ep}(\Sigma_{{w^{-1}(r)}}))\in \a_{\RP''}=\R^r.
\]
By using \eqref{rel2}, it is then easy to see that $\RP'=\RP_\Sigma$, and 
\be\label{sta0}
 \textrm{$w$ stabilizes the set }\left \{1+\sum_{i=1}^{k-1} d_i, 2+\sum_{i=1}^{k-1} d_i, \dots, \sum_{i=1}^k d_i\right \}
\ee
for all $k=1,2, \dots , r'$.

By \eqref{notzero2}, we know that  $A_{\RP''}$ acts on $\Sigma''$ by the character $e^{\lambda'}$. Together with the first equality in \eqref{lamdbdap}, this implies that  $A_{\RP_\Sigma}$ acts on $\Ind_{\RM(\A)}^{\RM_\Sigma(\A)}\Sigma''$ by the character $e^{\lambda}$. Finally it follows from \eqref{sta0} that $A_{\RP_\Sigma}$ acts on $\Ind_{\RM(\A)}^{\RM_\Sigma(\A)}\Sigma$ and $\Ind_{\RM(\A)}^{\RM_\Sigma(\A)}\Sigma''$ via the same character. Thus $\lambda=\lambda_\Sigma$. This finishes the proof of the lemma. 
\end{proof}

\begin{lemp}\label{lem26}
The composition of the  constant term map 
\[
\CA_{\RP_\Sigma}(\RG)\rightarrow \CA_{\RP}(\RG), \quad \phi\mapsto \left(\phi_{\RP}: g\mapsto \int_{\RN(\rk)\backslash \RN(\A)} \phi(xg) \od\! x\right)
\]
and the projection map (with respect to the decomposition \eqref{decomautform}) 
\[
     \CA_{\RP}(\RG)\rightarrow\CA_{\RP,\Sigma}(\RG)
  \]
induces an isomorphism 
\be\label{constermp}
 (\,\cdot\,)_{\RP,\Sigma}:  \CA_{\RP_\Sigma,[\Sigma]}^{\bar 2}(\RG)_{\lambda_\Sigma}\rightarrow  \CA_{\RP,\Sigma}(\RG):=\Ind_{\RP(\A)}^{\RG(\A)} (\C[\a_{\RP,\C}]\otimes \Sigma). 
  \ee
\end{lemp}

\begin{proof}
Write $\RP_1:=\RP\cap \RM_\Sigma=\RM\ltimes \RN_1$, where $\RN_1:=\RN\cap \RM_\Sigma$. 
The lemma is easily reduced to the following assertion: the  composition of the  constant term map 
\be\label{ct1}
\CA(\RM_\Sigma)\rightarrow \CA_{\RP_1}(\RM_\Sigma), \quad \phi\mapsto \left(\phi_{\RP_1 }: g\mapsto \int_{\RN_1(\rk)\backslash \RN_1(\A)} \phi(xg) \od\! x\right)
\ee
and the projection map
\be\label{pr1}
     \CA_{\RP_1}(\RM_\Sigma)\rightarrow\CA_{\RP_1,\Sigma}(\RM_\Sigma)
  \ee
induces an isomorphism 
   \be\label{iso2bar}
  \bigoplus_{[\Sigma']_{\RM_\Sigma}\subset [\Sigma]} \CA_{[\Sigma']_{\RM_\Sigma}}^{\bar 2}(\RM_\Sigma)_{\lambda_\Sigma}\cong \CA_{\RP_1,\Sigma}(\RM_\Sigma):=\Ind_{\RP_1(\A)}^{\RM_\Sigma(\A)} (\C[\a_{\RP,\C}]\otimes \Sigma). 
  \ee
  Here the direct sum is taken over all cuspidal data $[\Sigma']_{\RM_\Sigma}$ for $\RM_\Sigma$ such that $[\Sigma']_{\RM_\Sigma}\subset [\Sigma]$. 

  It is clear that 
  \[
  \CA_{[\Sigma']_{\RM_\Sigma}}^{\bar 2}(\RM_\Sigma)_{\lambda_\Sigma}\neq 0\Rightarrow [\Sigma']_{\RM_\Sigma}=[\Sigma]_{\RM_\Sigma}.
  \]
  This implies that the left hand side of \eqref{iso2bar} is equal to $\CA_{[\Sigma]_{\RM_\Sigma}}^{\bar 2}(\RM_\Sigma)$, and it remains to show that the composition of the constant term map \eqref{ct1} and the projection map \eqref{pr1} yields an isomorphism 
   \be\label{iso2bar1}
 \CA_{[\Sigma]_{\RM_\Sigma}}^{\bar 2}(\RM_\Sigma)\cong \CA_{\RP_1,\Sigma}(\RM_\Sigma). 
  \ee

  Without loss of generality, we assume that $r'=1$ so that $\RP_\Sigma=\RM_\Sigma=\RG$ and $\RP_1=\RP$. 
  Then by  the second tame condition and the description of the discrete spectrum for general linear groups  in \cite{MW89}, for every $\RP'\in \mathscr P$,  
\[
\CA^{2}_{\RP',[\Sigma]}(\RG)\neq \{0\}\quad \Rightarrow \quad \begin{cases}
    [\RP']=[\RP] \ \ (\textrm{that is, $\RP'$ is isomorphic to $\RP$ in $\mathscr P$}); &\smallskip\\
    \CA^{2,\mathrm{im}}_{\RP',[\Sigma]}(\RG)=\CA^{2}_{\RP',[\Sigma]}(\RG).&
\end{cases}
\]  
Therefore, by \eqref{dec2bar}, the Eisenstein series map yields an isomorphism 
 \[
    \CA^{2}_{\RP,[\Sigma]}(\RG)_{\Aut_{\mathscr P}(\RP)}\cong \CA^{\bar 2}_{[\Sigma]}(\RG).
    \]
 Note that by the  first tame  condition, the natural map
 \[
 \CA_{\RP,\Sigma}(\RG)=\CA^{2}_{\RP,\Sigma}(\RG)\rightarrow \CA^{2}_{\RP,[\Sigma]}(\RG)_{\Aut_{\mathscr P}(\RP)}
 \]
 is an isomorphism. Therefore by part (b) of Proposition \ref{lemti0}, \eqref{iso2bar1} follows. This finishes the proof of the lemma. 
\end{proof}

Now  we are ready to prove Theorem \ref{lemti00}. It follows from Proposition \ref{franke} and  Lemma \ref{lem25} that the map 
\[
\CA_{[\Sigma]}(\RG)\rightarrow \CA_{\RP_\Sigma,[\Sigma]}^{\bar 2}(\RG)_{\lambda_\Sigma}, \quad \phi\mapsto (\phi_{\RP_\Sigma})_{\lambda_\Sigma}
\]
is an isomorphism. As in \eqref{constermp}, we have an isomorphism 
\[
   (\,\cdot\,)_{\RP,\Sigma}: \CA_{\RP_\Sigma,[\Sigma]}^{\bar 2}(\RG)_{\lambda_\Sigma}\rightarrow  \CA_{\RP,\Sigma}(\RG). 
  \]
 The composition of the above two maps equals
\[
  \CA_{[\Sigma]}(\RG)\rightarrow  \CA_{\RP,\Sigma}(\RG), \quad \phi\mapsto  ((\phi_{\RP_\Sigma})_{\lambda_\Sigma})_{\RP,\Sigma}=\phi_{\RP,\Sigma}. 
\]
Therefore the second arrow in Theorem \ref{lemti00} is an isomorphism. Together with part (b) of Proposition \ref{lemti0}, this implies that the first arrow in Theorem \ref{lemti00} is also an isomorphism. This finishes the proof of Theorem  \ref{lemti00}.

\subsection{$\Aut(\C)$-action on $\CR^{\mathrm{coti}}$}

Let $\Pi$ be an isobaric automorphic representation of $\RG(\A)$ as in \eqref{isopi}. 
Now we assume that $\Pi$ is regular algebraic in the sense of \cite[Definition 3.12]{Clo}, namely the infinitesimal character of $\Pi_\infty$ coincides with that of an algebraic finite-dimensional representation of $\RG(\rk\otimes_\BQ \C)$. The following result generalizes \cite[Theorem 3.13]{Clo} and \cite[Lemma 1.2]{G18}. 

\begin{lemp}\label{sigmaiso}
   Let $\sigma\in \Aut(\C)$. Up to isomorphism, there exists a unique  regular algebraic isobaric automorphic representation $\,^{\sigma}\Pi$ of $\RG(\A)$ such that for all $v\nmid \infty$, $({}^{\sigma}\Pi)_v$ is isomorphic to the $\sigma$-twist of $\Pi_v$ (or equivalently, there exists a $\RG(\rk_v)$-equivariant $\sigma$-linear isomorphism $\Pi_v\rightarrow ({}^{\sigma}\Pi)_v$).
\end{lemp}
\begin{proof}
    The uniqueness follows from the strong multiplicity one theorem for isobaraic automorphism representations for general linear groups (see \cite[Theorem 4.4]{JS81b}). 
    
    To prove the existence, note that
$\Sigma_i\otimes\abs{\det}_\A^{\nu_i}$ is  regular algebraic and cuspidal, where 
$\nu_i$ is given in \eqref{nu_i}. By \cite[Theorem 3.13]{Clo}, we have the $\sigma$-twist $\,^\sigma \!\left(\Sigma_i\otimes\abs{\det}_\A^{\nu_i}\right)$. It is a regular algebraic irreducible cuspidal  automorphic representation of $\GL_{n_i}(\A)$ whose $v$-component is isomorphic to the $\sigma$-twist of $\Sigma_{i,v}\otimes\abs{\det}_v^{\nu_i}$. Define the $\sigma$-twist of $\Sigma_i$ to be
\[
  \,^{\sigma} \Sigma_i:=\,^\sigma \!\left(\Sigma_i\otimes\abs{\det}_\A^{\nu_i}\right) \otimes \abs{\det}_\A^{-\nu_i},
\]
which is also an irreducible subrepresentation of $\CA_{\mathrm{cusp}}(\GL_{n_i})$. It is clear that $\mathrm{ep}( \,^{\sigma} \Sigma_i)=\mathrm{ep}(\Sigma_i) $.

Using Theorem \ref{prop:std}, it is easy to see that the isobaric sum
\be\label{isobspi}
\,^\sigma \Pi:=\,^\sigma\Sigma_1\boxplus  \,^\sigma\Sigma_2\boxplus \dots\boxplus \,^\sigma\Sigma_r
\ee
satisfies the requirement of the lemma. 
\end{proof}



Let $\,^\sigma \Pi$ be as in Lemma \ref{sigmaiso}, to be called the $\sigma$-twist of $\Pi$. Then $(\sigma,\Pi)\mapsto \,^\sigma \Pi$ defines an action of $\Aut(\C)$ on the set of all  isomorphism classes of regular algebraic isobaric automorphic representations of $\RG(\A)$. 
By \cite[Theorem 3.13]{Clo} and the proof of Lemma \ref{sigmaiso}, this action is locally finite. 


\begin{prp}\label{prpactpi}
   Let $\sigma\in \Aut(\C)$. If $\Pi\in \CR^{\mathrm{coti}}$, then  $\,^\sigma \Pi\in \CR^{\mathrm{coti}}$ as well.
\end{prp}
\begin{proof}
    It is clear from the proof of Lemma \ref{sigmaiso} that $\,^\sigma \Pi$ is tamely isobaric. It remains to show that $\RI_\infty(\,^\sigma \Pi)$ is irreducible and cohomological. 
    To this end, we use the explicit realization of $\RI_\infty(\Pi)$ in Proposition \ref{class00} and \eqref{pigamma}. 
   
    Let $\CE_\rk$ be the set of field embeddings $\iota: \rk\to\BC$. For every dominant  weight 
    \[
    \mu = (\mu^\iota)_{\iota\in \CE_\rk} = (\mu^\iota_1,\mu^\iota_2,\ldots, \mu^\iota_n)_{\iota\in\CE_\rk}\in (\BZ^n)^{\CE_\rk},
    \]
    where $\mu^\iota_1\geq \mu^\iota_2\geq \cdots \geq\mu^\iota_n$, let $F_\mu\subset \CF$ be the  irreducible subrepresentation with highest weight $\mu$. 
    For every character $\varepsilon = \otimes_{v\mid\infty}\varepsilon_v: \pi_0(\rk_\infty^\times)\to\C^\times$, let $\pi_{\mu,\varepsilon}$ be the unique representation in  $\CR^{\rm coge,\infty}_\varepsilon$ with coefficient system $F_\mu$. 

   For every $\iota\in \CE_\rk$, put
   \[
   \tilde\mu^\iota_k := -\mu^\iota_k + k- \frac{n+1}{2}\qquad (k=1,2,\ldots, n).
   \]
   For every place $v\mid\infty$ of $\rk$, write
   \[
    \{\textrm{field embeddings $\rk\to\C$ inducing $v$}\} = 
        \{\iota_v, \bar\iota_v\}\subset \CE_\rk
     \]
   so that $\iota_v=\bar\iota_v$ if and only if $v$ is real, and 
   \[
   \CE_\rk = \bigcup_{v\mid\infty}\{\iota_v, \bar\iota_v\}.
   \]
   Here and below `$\bar{\ }$' indicates the complex conjugation. 
   By using $\iota_v$, $\rk_v$ is identified with
 \[
  \mathbb K_v:=  \begin{cases}
    \R,\quad& \textrm{if $v$ is real};\\
       \C,\quad& \textrm{if $v$ is complex}.
   \end{cases}
   \]
  
  Recall  the notation of Section \ref{secgc}. By Proposition \ref{class00}, we have $\pi_{\mu,\varepsilon} \cong \pi_\mu\otimes \varepsilon$, where 
   \[
    \pi_\mu = \widehat\otimes_{v\mid\infty}\pi(\gamma_v)\quad  \textrm{and} \quad  \gamma_v: = \Set{(\tilde\mu^{\iota_v}_1, \tilde\mu^{\bar\iota_v}_n), (\tilde\mu^{\iota_v}_2, \tilde\mu^{\bar\iota_v}_{n-1}),\ldots, (\tilde\mu^{\iota_v}_n, \tilde\mu^{\bar\iota_v}_1) } \in \Gamma_n^{\K_v}.
   \]
   Now assume that $\RI_\infty(\Pi)\cong \pi_{\mu,\varepsilon}$. By Proposition \ref{class}, for every place $v\mid \infty$ of $\rk$ there exists $\gamma_{i,v}\in \Gamma_{n_i}^{\K_v}$ ($i=1,2,\ldots, r$) such that 
   \[
   \Sigma_{i,v} \cong \pi(\gamma_{i,v}')_{\varepsilon_v} 
   \]
   and 
   \be\label{gammav}
    \gamma_{1,v}'\cup\gamma_{2,v}'\cup\cdots \cup \gamma_{r, v}' =\gamma_v\in \Gamma^{\K_v}_n,
   \ee
   where $\gamma_{i,v}'$ is defined as in \eqref{gammaip}.


For every $i=1,2,\ldots,r$, write 
  \[
\left \{ a^\iota_{i,j}\in \Z+\frac{n-1}{2}\right\}_{j=1,2,\ldots, n_i, \ \iota\in \CE_\rk}
  \]
  for the family 
  such that
  \begin{itemize}
\item  $\gamma_{i,v}' = \Set{(a_{i,1}^{\iota_v}, a_{i,n_i}^{\bar\iota_v}), (a_{i,2}^{\iota_v}, a_{i,n_i-1}^{\bar\iota_v}),\ldots, (a_{i,n_i}^{\iota_v}, a_{i,1}^{\bar\iota_v})}$ for all $v\mid \infty$; and 

 \item  $a_{i,1}^{\iota}> a_{i,2}^{\iota} > \cdots > a_{i,n_i}^{\iota}$ for all $\iota\in\CE_\rk$.
   \end{itemize}
   By the purity lemma (\cite[Lemma 4.9]{Clo}) we have that 
   \be \label{pure}
   a^{\sigma^{-1}\circ \bar\iota}_{i,j} = a^{\overline{\sigma^{-1}\circ\iota}}_{i,j},\quad \textrm{for all }\ \  j=1,2,\ldots, n_i, \ \iota\in\CE_\rk.
   \ee

Let $v\mid \infty$ be an Archimedean place of $\rk$. 
  For all  $i=1,2,\ldots, r$, define 
   \[
   \,^\sigma\gamma'_{i, v}:=\Set{(a_{i,1}^{\sigma^{-1}\circ \iota_v}, a_{i,n_i}^{\sigma^{-1}\circ \bar\iota_v}), (a_{i,2}^{\sigma^{-1}\circ \iota_v}, a_{i,n_i-1}^{\sigma^{-1}\circ \bar\iota_v}),\ldots, (a_{i,n_i}^{\sigma^{-1}\circ\iota_v}, a_{i,1}^{\sigma^{-1}\circ \bar\iota_v})}.
   \]
   It follows from \eqref{isobspi} that 
   \[
   \RI_\infty(\,^\sigma\Pi)_v\cong \pi(\,^\sigma\gamma_{1,v}')_{\varepsilon_v}\times \cdots \times \pi(\,^\sigma\gamma_{r,v}')_{\varepsilon_v}.
   \]
   
   By Proposition \ref{class} again, it suffices to show that 
   \be \label{incompare}
{}^\sigma\gamma'_{1,v}\cup{}^\sigma\gamma'_{2,v}\cup\cdots\cup \,^\sigma\gamma'_{r, v} \in \Gamma_n^{\K_v}.
   \ee
    Now we consider complex and real places separately. 
   \begin{itemize}
     \item If $v$ is complex, then \eqref{gammav} and \eqref{pure} imply \eqref{incompare}.

 \item Assume that $v$ is real. Then ${}^\sigma\gamma'_{1,v}\cup\cdots\cup \,^\sigma\gamma'_{r,v}$ is symmetric.
 \begin{itemize}
     \item 
If $\sigma^{-1}\circ \iota_v$ is real, then \eqref{incompare} is clear from \eqref{gammav}.
 \item If $\sigma^{-1}\circ \iota_v$ is complex, then \eqref{pure} gives
 \[
a^{\sigma^{-1} \circ\iota_v}_{i,j} = a^{\sigma^{-1}\circ\bar\iota_v}_{i,j} = a^{\overline{\sigma^{-1}\circ \iota_v}}_{i,j},
 \]
which together with \eqref{gammav} implies \eqref{incompare}. 
 \end{itemize}
\end{itemize}
Thus we obtain \eqref{incompare} and thereby finish the proof of the proposition. 
\end{proof}

In view of Proposition \ref{prpactpi}, we get a locally finite action
\[
 \Aut(\C)\curvearrowright \CR^{\mathrm{coti}}, \quad (\sigma, \Pi)\mapsto \,^\sigma \Pi.
\]

\subsection{$\Aut(\C)$-action on the Betti cohomology}

Suppose that $\Pi\in\CR^{\mathrm{coti}}$. 
As defined in the Introduction, we have the cohomology space
\[
\RH^b(\RI^{\mathrm{aut}}(\Pi)):=\RH_{\mathrm{ct}}^b(\RG(\rk_\infty)^0; F_{\Pi}\otimes \RI^{\mathrm{aut}}(\Pi))=\RH^b(\g,K_\infty^0; F_{\Pi}\otimes \RI^{\mathrm{aut}}(\Pi)).
\]
Here $\g$ denotes the complexified Lie algebra of $\RG(\rk_\infty)$.

\begin{lemp}\label{injbetti00}
 The natural map 
 \[
  \RH^b(\RI^{\mathrm{aut}}(\Pi)) \rightarrow \RH^b(\g,K_\infty^0; F_{\Pi}\otimes \CA_{[\Sigma]}(\RG))
  \]
 is an isomorphism of representations of $\pi_0(\rk_\infty^\times)\times \RG(\A_\mathrm f)$. 
 

\end{lemp}

\begin{proof}
In view of Theorem \ref{lemti00}, the lemma is equivalent to saying that the natural map
\[
  \RH^b(\g,K_\infty^0; F_{\Pi}\otimes \Ind_{\RP(\A)}^{\RG(\A)} \Sigma) 
  \rightarrow \RH^b(\g,K_\infty^0; F_{\Pi}\otimes \CA_{\RP,\Sigma}(\RG))
  \]
 is an isomorphism (see \eqref{autp}). This holds true because $b$ is the minimal degree of non-vanishing cohomology (see the proof of \cite[Proposition 1.6]{G18}). 
\end{proof}

By Lemma \ref{injbetti00}, the inclusion maps $\RI^{\mathrm{aut}}(\Pi)\hookrightarrow \CA(\RG)$ and $F_\Pi\hookrightarrow \CF$ induce an injective linear map 
\[
  \RH^b(\RI^{\mathrm{aut}}(\Pi))\rightarrow \RH^b(\CX).
\]
Using this injective map, we identify 
$\RH^b(\RI^{\mathrm{aut}}(\Pi))$ as a subspace of $\RH^b(\CX)$. 

\begin{lemp}\label{actbetti}
   For every $\sigma\in \Aut(\C)$, 
    \[
\,^\sigma(\RH^b(\RI^{\mathrm{aut}}(\Pi)))= \RH^b(\RI^{\mathrm{aut}}(\,^\sigma\Pi)).
    \]
\end{lemp}

\begin{proof}
    This follows from Lemma  \ref{injbetti00} and \cite[Theorem 4.3]{FS}.  
\end{proof}

In view of Lemma \ref{actbetti}, the action \eqref{actbetti000} induces a sesquilinear action
 \be\label{actfinite1}
\Aut(\C) \curvearrowright \bigsqcup_{\Pi\in  \CR^{\mathrm{coti}} }  \RH^b(\RI^{\mathrm{aut}}(\Pi)),
    \ee
    which is smooth by Lemma \ref{ratk}.

\subsection{Betti-Whittaker periods} \label{sec:actf}

We define Betti-Whittaker periods following \cite[Sections 3.3, 3.4]{Mah} and \cite{RS08}.



Suppose $\Pi\in \CR^{\mathrm{coti}}_\varepsilon$
so that 
\[
\dim \RH^b_\varepsilon(\RI_\infty^{\mathrm{whi}}(\Pi))=1. 
\]
Then Lemma \ref{lem20} implies that 
\be\label{mulone2}
 \dim \Hom_{\RG(\A_\mathrm f)}(\RH_\varepsilon^b(\RI^{\mathrm{aut}}(\Pi)), \CW_\mathrm f)= 1.
\ee
The actions \eqref{actfinite1} and \eqref{actwhi00} induce a sesquilinear action 
\be\label{acthom}
  \Aut(\C)\curvearrowright \bigsqcup_{\Pi\in  \CR_\varepsilon^{\mathrm{coti}} }  \Hom_{\RG(\A_\mathrm f)}(\RH_\varepsilon^b(\RI^{\mathrm{aut}}(\Pi)), \CW_\mathrm f).
\ee
We prove the following theorem, see also \cite[Section 3.4.1]{Mah}.

\begin{thmp}\label{smooth2}
    The sesquilinear action \eqref{acthom} is smooth. 
\end{thmp}
\begin{proof}
   We first claim that there exists a linear functional 
    \be\label{ratinallambda}
    0\neq \lambda\in \Hom_{\RN(\A_\mathrm f)}(\RH_\varepsilon^b(\RI^{\mathrm{aut}}(\Pi)), \psi_{\RN,\mathrm f})
    \ee
    that takes values in $\overline \BQ$ on the $\BQ[\Pi]$-form $\RH_\varepsilon^b(\RI^{\mathrm{aut}}(\Pi))^{\Aut(\C/\BQ[\Pi])}$.
    This claim follows from  the equality 
    \[
    \left(\RH_\varepsilon^b(\RI^{\mathrm{aut}}(\Pi))^{\Aut(\C/\overline \BQ)}\otimes_{\overline \BQ} \psi_{\RN,\mathrm f,\overline \BQ}^{-1}\right)_{\RN(\A_\mathrm f)}\otimes_{\overline \BQ} \C=\left(\RH_\varepsilon^b(\RI^{\mathrm{aut}}(\Pi))\otimes_\C \psi_{\RN,\mathrm f}^{-1}\right)_{\RN(\A_\mathrm f)}, 
    \]
    where $\psi_{\RN,\mathrm f,\overline \BQ}: \RN(\A_\mathrm f)\rightarrow \overline \BQ^\times$ is the character $\psi_{\RN,\mathrm f}$ viewed as a $\overline \BQ^\times$-valued character, and the subscript group indicates the coinvariant space. 

    Fix an element $\lambda$ as in \eqref{ratinallambda}, which corresponds to an element 
    \[
    0\neq \Lambda\in \Hom_{\RG(\A_\mathrm f)}(\RH_\varepsilon^b(\RI^{\mathrm{aut}}(\Pi)), \CW_\mathrm f)
    \]
    by Frobenius reciprocity. We claim that $\Lambda$ is smooth with respect to the sesquilinear action \eqref{acthom}. Fix a vector $u \in \RH_\varepsilon^b(\RI^{\mathrm{aut}}(\Pi))^{\Aut(\C/\BQ[\Pi])}$ such that \[
    0\neq \la \lambda, u \ra \in \overline \BQ. 
\]

    To prove the claim, let $\sigma\in \Aut(\C/\BQ[\Pi])$. Then the uniqueness  property \eqref{mulone2} implies that 
    \[
    \sigma. \Lambda= a_\sigma\cdot  \Lambda\quad \textrm{for some }a_\sigma\in \C^\times.  
    \]
Now we suppose that $\sigma$ is in a sufficiently small open subgroup of $\Aut(\C/\BQ[\Pi])$ so that
\[
\mathbf t_\sigma^{-1}.(\Lambda(u ))=\Lambda(u )\quad\textrm{and}\quad \sigma(\la \lambda,  u \ra)=\la \lambda,  u \ra.
\]
Then 
\begin{eqnarray*}
\la\lambda_{\psi,\mathrm f}, (\sigma. \Lambda)(u )\ra  
&=&\la \lambda_{\psi,\mathrm f}, \,^\sigma\left(\Lambda(^{\sigma^{-1}} u )\right)\ra \\
&=&\la \lambda_{\psi,\mathrm f}, \,^\sigma\left(\Lambda(u)\right)\ra \\
&=&\sigma((\Lambda(u )_\psi)(\mathbf t_\sigma^{-1}))\\
&=&\sigma(\la \lambda_{\psi,\mathrm f}, \mathbf t_\sigma^{-1}.(\Lambda(u ))\ra) \\
&=&\sigma(\la \lambda,  u\ra) \\
&=&\la \lambda,  u \ra.
\end{eqnarray*}
Here $\Lambda(u )_\psi\in \Ind_{\RN(\A_\mathrm f)}^{\RG(\A_\mathrm f)}\psi_{\RN,\mathrm f}$ is the $\psi$-component of $\Lambda(u )\in \CW_\mathrm f$. 

On the other hand
\begin{eqnarray*}
 \la\lambda_{\psi,\mathrm f}, (\sigma. \Lambda)(u )\ra  
&=&a_\sigma\cdot \la\lambda_{\psi,\mathrm f}, \Lambda(u )\ra   \\
&=&a_\sigma\cdot \la \lambda,  u \ra.
\end{eqnarray*}
Therefore $a_\sigma=1$ and the theorem follows. 
\end{proof}

\begin{remarkp}
      Theorem \ref{smooth2} has the following consequence: the action \eqref{actwhi00} induces a well-defined smooth sesquilinear action  
 \[
\Aut(\C) \curvearrowright \bigsqcup_{\Pi\in  \CR^{\mathrm{coti}} }  \RI^\mathrm{whi}_\mathrm{f}(\Pi).
    \]  
\end{remarkp}

\subsection{$\Aut(\C)$-action on the Archimedean cohomology}

Suppose $\Pi\in \CR^{\mathrm{coti}}_\varepsilon$ as before. 
The homomorphism 
\[
\jmath:  \CA(\RG)\rightarrow 
\CW
\]
restricts to a homomorphism 
\[
\jmath: \RI^{\mathrm{aut}}(\Pi)\rightarrow \RI_\infty^{\mathrm{whi}}(\Pi)\otimes \CW_\mathrm f,
\]
which further induces a homomorphism 
\[
\jmath: \RH^b_{\varepsilon}(\RI^{\mathrm{aut}}(\Pi))\rightarrow \RH^b_{\varepsilon}(\RI_\infty^{\mathrm{whi}}(\Pi))\otimes \CW_\mathrm f. 
\]
The last homomorphism induces a linear isomorphism 
\be\label{jmathiso}
\jmath:  \RH^b_{\varepsilon}(\RI_\infty^{\mathrm{whi}}(\Pi))^\vee\xrightarrow{\sim}  \Hom_{\RG(\A_\mathrm f)}(\RH^b_{\varepsilon}(\RI^{\mathrm{aut}}(\Pi)), \CW_\mathrm f) 
\ee
of one-dimensional vector spaces, where a superscript `$^\vee$' indicates the dual space.
Using these isomorphisms for various $\Pi$, the smooth sesquilinear action \eqref{acthom} induces a smooth sesquilinear action 
\[
  \Aut(\C)\curvearrowright \bigsqcup_{\Pi\in  \CR_\varepsilon^{\mathrm{coti}} }   \RH^b_{\varepsilon}(\RI_\infty^{\mathrm{whi}}(\Pi))^\vee.
\]

As in the Introduction,  by using Lemma \ref{sectionpi00}, we fix an $\Aut(\C)$-equivariant section
\be\label{varpi}
\varpi: \CR_\varepsilon^{\mathrm{coti}}\rightarrow  \bigsqcup_{\Pi\in  \CR_\varepsilon^{\mathrm{coti}} } \left(\RH^b_\varepsilon(\RI_\infty^{\mathrm{whi}}(\Pi))^\vee\setminus \{0\}\right),
\ee
and define the Betti-Whittaker period 
    \[
   \Omega_\varepsilon(\Pi) := \Omega_\varepsilon(\Pi; \varpi, \kappa_\varepsilon):=\la \varpi(\Pi), \kappa_{\varepsilon,\Pi}\ra. 
    \]

\section{MVW-involutions} 

\subsection{MVW-involutions on tamely isobaric automorphic  representations}
Following Moeglin, Vign\'eras, and  Waldspurger (\cite{MVW87}), we define the MVW-involution of $\RG$ to be the algebraic automorphism 
\be\label{mvw0}
 \breve{\ } : \RG\rightarrow \RG, \quad g \mapsto  \breve g:=w_n g^{-t} w_n^{-1}, 
\ee
where the superscript `$-t$' indicates the inverse transpose of a matrix, and 
\[
w_n:=\begin{bmatrix}
    & & & &  1\\
    & & & -1 &\\
    & &  1 & \\
    &\iddots & & \\
    (-1)^{n-1}
\end{bmatrix}\in\GL_n(\Z).
\]

      It is well known and easy to see that 
      \be\label{conj}
        \textrm{ for all $g\in \RG(\rk')$, $\breve g$ is conjugate to $g^{-1}$ in $\RG(\rk')$,}
      \ee
      where $\rk'$ is a field extension of $\rk$. 
      
      The MVW-involution  induces an involutive  map 
      \be\label{mvwcd}
      \breve{\ } : \{\textrm{cuspidal data for $\RG$}\}\rightarrow \{\textrm{cuspidal data for $\RG$}\}. 
      \ee
In view of Harish-Chandra's character theory (see \cite{HC65a, HC65b}), the property \eqref{conj} implies that the above map sends every cuspidal datum $[(\RM, \Sigma)]$ to $[(\RM, \Sigma^\vee)]$, where $\Sigma^\vee$ denotes the irreducible subrepresentation of $\CA_{\mathrm{cusp}}(\RM)$ that is isomorphic to the contragredient of $\Sigma$. 

Since the set of cuspidal data for $\RG$ is naturally in bijection with the set $\CR^{\mathrm{isob}}$ of isomorphism classes of isobaric automorphic representations of $\RG(\A)$, the involutive map \eqref{mvwcd} yields an involutive map 
\be\label{mvwcd2}  
\breve{\ } :\CR^{\mathrm{isob}}\rightarrow \CR^{\mathrm{isob}}, \quad \Pi\mapsto \breve \Pi.
\ee
    Similarly, for every $\Pi\in \CR^{\mathrm{isob}}$, $\breve \Pi$ is isomorphic to the contragredient of $\Pi$. 

 The MVW-involution also  induces an involution
    \be\label{mvw06} 
\breve{\ }:\CA(\RG)\rightarrow \CA(\RG), \quad \phi\mapsto \breve \phi.
\ee
For every  subrepresentation $\Pi_1\subset \CA(\RG)$, we define its MVW-involution 
\[
  \breve \Pi_1:=(\Pi_1)\breve\, :=\{\breve  \phi \,:\, \phi\in \Pi_1\}\subset \CA(\RG).
\]

\begin{lemp}\label{lemmvw24}
The equality
\[
(\RI^{\mathrm{aut}}(\Pi))\breve\, = \RI^{\mathrm{aut}}(\breve \Pi)
\]
holds for all tamely isobaric automorphic representations $\Pi$ of $\RG(\A)$.
\end{lemp}
\begin{proof}
  Let $\breve \RP$ denote the upper-triangular parabolic subgroup of $\RG$ with Levi factor $\breve \RM:=\GL_{n_r}\times \GL_{n_{r-1}}\times \dots \times \GL_{n_1}$. Applying the MVW-involution to $\Sigma$, we obtain an irreducible subrepresentation $\breve \Sigma\subset \CA_{\mathrm{cusp}}(\breve \RM)$.  It is routine to check that the diagram \[
\begin{CD}
\mathrm{Ind}^{\RG(\A)}_{\RP(\A)}\Sigma @>\CE >>\CA(\RG)\\
    @V \breve \   VV @VV \breve\     V\\
\mathrm{Ind}^{\RG(\A)}_{\breve \RP(\A)}\breve \Sigma @>\CE >>\CA(\RG)
\end{CD}
\]
commutes, where the horizontal arrows are defined as in \eqref{ce00}, and the left vertical arrow is the map induced by the MVW-involution. This proves the lemma, in view of Remark \ref{iaut}. 
\end{proof}

It is clear that the map \eqref{mvwcd2} restricts to an involutive map 
\be\label{mvwoncoti}
\breve{\ }: \CR_\varepsilon^{\mathrm{coti}}\rightarrow \CR_\varepsilon^{\mathrm{coti}}, \quad \Pi\mapsto \breve \Pi.
\ee

The following lemma is easily verified. 
\begin{lemp}\label{lemaute0}
    The map \eqref{mvwoncoti} is $\Aut(\C)$-equivariant. 
\end{lemp}

\subsection{MVW-involutions on the Whittaker spaces} \label{sec:mvwwhit}
Note that the MVW-involution stabilizes $\RN$ and the character $\psi_{\RN}$. Thus it induces an involution 
\be\label{mvw07700} 
\breve{\ }:\Ind_{\RN(\BA)}^{\RG(\BA)} \psi_{\RN}\rightarrow \Ind_{\RN(\BA)}^{\RG(\BA)} \psi_{\RN}, \quad f\mapsto \breve f. 
\ee
Denote by 
\be\label{mvwwinfty0000}
\imath_{\psi}: \CW\rightarrow \CW
\ee
the involution that makes the diagram
\[
\begin{CD}
  \CW @>\imath_{\psi} \ >>\CW\\
    @V \textrm{projection} VV @VV   \textrm{projection} V\\
\Ind_{\RN(\A)}^{\RG(\A)} \psi_{\RN} @>\breve \  >>\Ind_{\RN(\A)}^{\RG(\A)} \psi_{\RN}
\end{CD}
\]
commute. 

Note that for each subrepresentation $\tau$ of $\CW$, 
\[
\breve \tau:=\imath_\psi(\tau)
\]
is  a subrepresentation of $\CW$ which is independent of $\psi$. 

Similarly, the MVW-involution  induces involutions 
\be\label{mvw07701} 
\breve{\ }:\Ind_{\RN(\rk_\infty)}^{\RG(\rk_\infty)} \psi_{\RN,\infty}\rightarrow \Ind_{\RN(\rk_\infty)}^{\RG(\rk_\infty)} \psi_{\RN,\infty}, \quad f\mapsto \breve f
\ee
and
\be\label{mvw07702} 
\breve{\ }:\Ind_{\RN(\BA_\mathrm f)}^{\RG(\BA_\mathrm f)} \psi_{\RN, \mathrm f}\rightarrow \Ind_{\RN(\BA_\mathrm f)}^{\RG(\BA_\mathrm f)} \psi_{\RN,\mathrm f}, \quad f\mapsto \breve f. 
\ee
Similar to \eqref{mvwwinfty0000}, we have involutions
\be\label{mvwwinfty0001}
\imath_{\psi}: \CW_\infty\rightarrow \CW_\infty
\quad\textrm{and}\quad 
\imath_{\psi}: \CW_\mathrm f\rightarrow \CW_\mathrm f.
\ee
For all subrepresentations $\tau_\infty\subset \CW_\infty$ and  $\tau_\mathrm f\subset \CW_\mathrm f$, 
\[
\breve \tau_\infty:=\imath_{\psi}(\tau_\infty)\quad \textrm{and }\quad \breve \tau_\mathrm f:=\imath_{\psi}(\tau_{\rm f})
\]
are respectively subrepresentations of $\CW_\infty$ and $\CW_\mathrm f$, which are independent of $\psi$.

\begin{lemp}\label{leminvf}
 Let $\tau\subset \CW_\mathrm f$ be a subrepresentation with central character $\omega'$. Then for every $\sigma\in \Aut(\C)$, 
 \[
\imath_{\psi}(\sigma(\tau))=\sigma(\breve \tau)
\] 
and the diagram 
    \[
\begin{CD}
   \tau @>\sigma >>  \sigma(\tau)\\
    @V \mathscr{G}(\omega',\psi)^{1-n}\cdot\imath_{\psi}     VV @VV\mathscr{G}(\sigma\circ\omega', \psi)^{1-n}\cdot\imath_{\psi} V\\
  \breve \tau @>\sigma >> \sigma(\breve \tau)
\end{CD}
\]
commutes. 
\end{lemp}
\begin{proof}
 Let $f\in \tau$, to be viewed as an element of $\Ind^{\RG(\A_\mathrm f)}_{{\rm N}(\A_\mathrm f)}\psi_{\RN,\mathrm f}$ via the projection map $\CW_\mathrm f\rightarrow \Ind^{\RG(\A_\mathrm f)}_{{\rm N}(\A_\mathrm f)}\psi_{\RN,\mathrm f}$.  For every $x\in \RG(\A_\mathrm f)$, by \eqref{gauss} we have
 \begin{eqnarray*}
   && \left((\sigma\circ (\mathscr{G}(\omega',\psi)^{1-n}\cdot\imath_{\psi})). f\right)(x)\\
   &=& \sigma\left( (\mathscr{G}(\omega',\psi)^{1-n}\cdot\imath_{\psi}). f)(\mathbf t_\sigma^{-1} x)\right)\\
    &=& \sigma\left( \mathscr{G}(\omega',\psi)^{1-n}\cdot f(\breve{\mathbf t}_\sigma^{-1}\breve x)\right)\\
     &=& \mathscr{G}( \sigma\circ \omega',\psi)^{1-n}\cdot \sigma\left( \omega'(t_{\sigma,\A}^{1-n}) \cdot f(\breve{\mathbf t}_\sigma^{-1}\breve x)\right)\\
      &=& \mathscr{G}(\sigma\circ \omega',\psi)^{1-n}\cdot \sigma\left( f(t_{\sigma,\A}^{1-n}\breve{\mathbf t}_\sigma^{-1}\breve x)\right)\\
      &=& \mathscr{G}(\sigma\circ \omega',\psi)^{1-n}\cdot \sigma\left( f(\mathbf t_{\sigma}^{-1}\breve x)\right)\\
       &=& \left((\mathscr{G}(\sigma\circ \omega',\psi)^{1-n}\cdot \imath_{\psi})\circ \sigma).f \right)(x).
\end{eqnarray*}
This proves the lemma. 
\end{proof}

The following lemma is routine to verify, and the proof is omitted. 

 \begin{lemp}\label{vw}
   The diagram 
\[
\begin{CD}
  \CA(\RG) @>\breve \  >>\CA(\RG)\\
    @V \jmath  VV @VV  \jmath  V\\
    \CW@>\imath_\psi\ >>\CW\\
\end{CD}
\]
commutes.
\end{lemp}

\subsection{MVW-involutions on the Archimedean cohomologies} \label{sec:mvwarch}

Similar to \eqref{mvwwinfty0000}, we have an involution 
\be\label{mvwcf}
\imath_{\psi}: \CF\rightarrow \CF
\ee
 that makes the diagram
\[
\begin{CD}
  \CF @>\imath_{\psi} \ >>\CF\\
    @V \textrm{projection to the $\psi$-component} VV @VV   \textrm{projection to the $\psi$-component} V\\
^{\mathrm{alg}}\Ind_{\RN(\rk\otimes_\BQ \C)}^{\RG(\rk\otimes_\BQ \C)} \C @>\breve \  >>^{\mathrm{alg}}\Ind_{\RN(\rk\otimes_\BQ \C)}^{\RG(\rk\otimes_\BQ \C)} \C \end{CD}
\]
commute. Here the bottom horizontal arrow is the involution induced by the MVW-involution. For each subrepresentation $F\subset \CF$, write $\breve F:=\imath_{\psi}(F)$, which is a subrepresentation of $\CF$ independent of $\psi$. 

Recall the representation 
\[
\pi_\varepsilon\in \CR_\varepsilon^{\mathrm{coge},\infty}\subset \CR^{\mathrm{coge},\infty}.
\]
Note that $\breve{\pi_\varepsilon}=\pi_\varepsilon$. 
Let $\pi\in \CR_\varepsilon^{\mathrm{coge},\infty}$. Then $\breve \pi\in \CR_\varepsilon^{\mathrm{coge},\infty}$, and the coefficient systems of $\pi$ and $\breve \pi$ are related by 
\[
\breve{ F_{\pi}}=F_{\breve \pi}.
\]

\begin{lemp}\label{mvwiotapi}
    The diagram 
\[
\begin{CD}
  \pi_\varepsilon @>\iota_\pi >>F_\pi\otimes \pi\\
    @V \imath_\psi  VV @VV  \imath_\psi\otimes \imath_\psi  V\\
   \pi_\varepsilon @>\iota_\pi >>F_{\breve \pi} \otimes \breve \pi
\end{CD}
\]
commutes.
\end{lemp}
\begin{proof}
   As in Lemma \ref{iotapi}, we have a linear functional $\lambda_\psi\otimes \lambda_\psi: F_{\breve \pi} \otimes \breve \pi\rightarrow \C$.  It is clear that the diagram 
\[
\begin{CD}
  \pi_\varepsilon @> \iota_\pi >>F_\pi\otimes \pi\\
    @V \imath_\psi  VV @VV  (\lambda_\psi\otimes \lambda_\psi)\circ(\imath_\psi\otimes \imath_\psi)  V\\
   \pi_\varepsilon @>(\lambda_\psi\otimes \lambda_\psi)\circ\iota_\pi >>\C 
\end{CD}
\]
commutes. This implies the lemma. 
\end{proof}

The MVW-involution $\breve \ : \RG(\rk\otimes_\BQ \C)\rightarrow \RG(\rk\otimes_\BQ \C)$, together with the 
involutions \eqref{mvwwinfty0000} and \eqref{mvwcf}, yields a linear isomorphism 
\be\label{imathinfty}
  \imath_\psi: \RH^b_\varepsilon(\pi)\rightarrow \RH^b_\varepsilon(\breve \pi).
\ee

\begin{thmp}
    \label{mvwpiepsilon}
    The map 
    \[
  \imath_\psi: \RH^b_\varepsilon(\pi_{\varepsilon})\rightarrow \RH^b_\varepsilon( \pi_{\varepsilon})
\]
is equal to the multiplication by 
\be \label{delta}
\delta_n : = (-1)^{r_1\cdot\frac{m(m+1)}{2}},\quad \text{where } m:=\left\lfloor \frac{n-1}{2}\right\rfloor,
\ee
and $r_1$ is the number of real places of $\rk$.
\end{thmp}

Theorem \ref{mvwpiepsilon} will be proved in the next section, as a consequence of Theorem \ref{mvwpiK}.

\begin{corp}
  \label{mvwkappa}
    The map  
  \[
 \imath_\psi : \RH^b_\varepsilon(\pi)  \rightarrow \RH^b_\varepsilon(\breve \pi)
\]
sends $\kappa_{\varepsilon, \pi}$ to $\delta_n \cdot \kappa_{\varepsilon, \breve \pi}$, where $\delta_n$ is in \eqref{delta}.   
\end{corp}
\begin{proof}
Lemma \ref{mvwiotapi} implies that the diagram 
\[
\begin{CD}
  \RH^b_\varepsilon(\pi_\varepsilon) @>\iota_\pi >>\RH^b_\varepsilon(\pi)\\
    @V \imath_\psi  VV @VV  \imath_\psi  V\\
\RH^b_\varepsilon(\pi_\varepsilon)  @>\iota_\pi >>\RH^b_\varepsilon(\breve \pi)
\end{CD}
\]
commutes. Therefore the corollary follows from Theorem \ref{mvwpiepsilon}. 
\end{proof}

The map in \eqref{imathinfty} yields a bundle involution 
\[
\imath_\psi: \bigsqcup_{\Pi\in  \CR_\varepsilon^{\mathrm{coti}} }  \RH^b_{\varepsilon}(\RI_\infty^{\mathrm{whi}}(\Pi))\rightarrow \bigsqcup_{\Pi\in  \CR_\varepsilon^{\mathrm{coti}} }  \RH^b_{\varepsilon}(\RI_\infty^{\mathrm{whi}}(\Pi)) 
\]
over the involution $\breve\ : \CR_\varepsilon^{\mathrm{coti}}\rightarrow \CR_\varepsilon^{\mathrm{coti}}$. This further induces a bundle involution 
\be\label{invimath2}
\imath_\psi: \bigsqcup_{\Pi\in  \CR_\varepsilon^{\mathrm{coti}} }  \RH^b_{\varepsilon}(\RI_\infty^{\mathrm{whi}}(\Pi))^\vee\rightarrow \bigsqcup_{\Pi\in  \CR_\varepsilon^{\mathrm{coti}} }  \RH^b_{\varepsilon}(\RI_\infty^{\mathrm{whi}}(\Pi))^\vee 
\ee
over $\breve\ : \CR_\varepsilon^{\mathrm{coti}}\rightarrow \CR_\varepsilon^{\mathrm{coti}}$,
which is specified by
\[
\la \imath_\psi(\phi),u\ra=\la \phi, \imath_\psi(u)\ra
\]
for all $\Pi\in  \CR_\varepsilon^{\mathrm{coti}}$, $\phi\in \RH^b_{\varepsilon}(\RI_\infty^{\mathrm{whi}}(\Pi))^\vee$,  and $u\in \RH^b_{\varepsilon}(\RI_\infty^{\mathrm{whi}}(\breve \Pi))$. 

\subsection{MVW-involutions on the Betti  cohomologies}
 Let $\Pi\in \CR^{\mathrm{coti}}_\varepsilon$. The linear maps
\[
\breve \ : \RI^{\mathrm{aut}}(\Pi)\rightarrow \RI^{\mathrm{aut}}(\breve \Pi)\qquad (\textrm{see Lemma \ref{lemmvw24}})
\]
and 
\[
\imath_\psi: F_{\Pi}\rightarrow F_{\breve \Pi},
\]
together with the MVW-involution $\breve\ : \RG(\rk_\infty)\rightarrow \RG(\rk_\infty) $,
induce a linear map
\be\label{invhaut1}
\imath_\psi: \RH^b(\RI^{\mathrm{aut}}(\Pi))\rightarrow \RH^b(\RI^{\mathrm{aut}}(\breve \Pi)). 
\ee
Putting these linear maps for various $\Pi$ together, we get an involution
\be\label{mvwco1}
\imath_\psi: \bigsqcup_{\Pi\in  \CR_\varepsilon^{\mathrm{coti}} } \RH_\varepsilon^b(\RI^{\mathrm{aut}}(\Pi))\rightarrow \bigsqcup_{\Pi\in  \CR_\varepsilon^{\mathrm{coti}} } \RH_\varepsilon^b(\RI^{\mathrm{aut}}(\Pi)).
 \ee
\begin{lemp}\label{leminvf2}
    The involution \eqref{mvwco1} is $\Aut(\C)$-equivariant. 
\end{lemp}
 \begin{proof}
    The MVW-involution induces an involution
\[
\breve{\ }: \CX\rightarrow \CX.
\]
Together with the involution $\imath_\psi: \CF\rightarrow \CF$, this induces an involution
\be\label{invcx}
  \imath_\psi: \RH^b(\CX)\rightarrow \RH^b(\CX). 
\ee
Because $\imath_\psi: \CF\rightarrow \CF$ is   $\Aut(\C)$-equivariant, the involution  \eqref{invcx} is also $\Aut(\C)$-equivariant. 
The lemma then follows by noting  that the diagram
\[
\begin{CD}
\RH_\varepsilon^b(\RI^{\mathrm{aut}}(\Pi)) @>\subset >>\RH^b(\g,K_\infty^0; \CF\otimes \CA(\RG))@> = >>\RH^b(\CX)\\
    @V \imath_\psi VV @VV    V @VV \imath_\psi V\\
\RH_\varepsilon^b(\RI^{\mathrm{aut}}(\breve \Pi)) @>\subset >>\RH^b(\g,K_\infty^0; \CF\otimes \CA(\RG))@> = >>\RH^b(\CX)
\end{CD}
\]
is commutative, where the middle vertical arrow is the map induced by the MVW-involution, as in \eqref{invhaut1}.
 \end{proof}

Recall from \eqref{jmathiso} the linear isomorphism 
\be\label{jmathiso2}
\jmath:  \RH^b_{\varepsilon}(\RI_\infty^{\mathrm{whi}}(\Pi))^\vee\xrightarrow{\sim}  \Hom_{\RG(\A_\mathrm f)}(\RH^b_{\varepsilon}(\RI^{\mathrm{aut}}(\Pi)), \CW_\mathrm f) 
\ee
of one-dimensional vector spaces. 
 Putting  together these linear isomorphisms  for various $\Pi$, we get a bundle isomorphism 
\[
\jmath:  \bigsqcup_{\Pi\in  \CR_\varepsilon^{\mathrm{coti}} }  \RH^b_{\varepsilon}(\RI_\infty^{\mathrm{whi}}(\Pi))^\vee \xrightarrow{\sim} \bigsqcup_{\Pi\in  \CR_\varepsilon^{\mathrm{coti}} } \Hom_{\RG(\A_\mathrm f)}(\RH^b_{\varepsilon}(\RI^{\mathrm{aut}}(\Pi)), \CW_\mathrm f). 
\]
By using Lemma \ref{vw}, it is routine to check that the diagram
\be\label{j}
\begin{CD}
  \bigsqcup_{\Pi\in  \CR_\varepsilon^{\mathrm{coti}} }  \RH^b_{\varepsilon}(\RI_\infty^{\mathrm{whi}}(\Pi))^\vee   @>\jmath >>\bigsqcup_{\Pi\in  \CR_\varepsilon^{\mathrm{coti}} } \Hom_{\RG(\A_\mathrm f)}(\RH^b_{\varepsilon}(\RI^{\mathrm{aut}}(\Pi)), \CW_\mathrm f) \\
    @V \imath_\psi \ (\textrm{see \eqref{invimath2})} VV @VV   \imath_\psi V\\
 \bigsqcup_{\Pi\in  \CR_\varepsilon^{\mathrm{coti}} }  \RH^b_{\varepsilon}(\RI_\infty^{\mathrm{whi}}(\Pi))^\vee   @>\jmath >>\bigsqcup_{\Pi\in  \CR_\varepsilon^{\mathrm{coti}} } \Hom_{\RG(\A_\mathrm f)}(\RH^b_{\varepsilon}(\RI^{\mathrm{aut}}(\Pi)), \CW_\mathrm f) 
\end{CD}
\ee
commutes, where the right vertical arrow $\imath_\psi$ is the bundle map over $\breve\ : \CR_\varepsilon^{\mathrm{coti}}\rightarrow \CR_\varepsilon^{\mathrm{coti}}$
 specified by
\[
(\imath_\psi(\phi))(u)=\imath_\psi( \phi( \imath_\psi(u)))
\]
for all $\Pi\in  \CR_\varepsilon^{\mathrm{coti}}$, $\phi\in \Hom_{\RG(\A_\mathrm f)}(\RH^b_{\varepsilon}(\RI^{\mathrm{aut}}(\Pi)), \CW_\mathrm f)$,  and $u\in \RH^b_{\varepsilon}(\RI^{\mathrm{aut}}(\breve \Pi))$. 

Define a function 
\be\label{gpsi}
\CG_\psi: \CR_\varepsilon^{\mathrm{coti}}\rightarrow \C^\times, \quad \Pi\mapsto \mathscr{G}(\omega_\Pi,\psi)^{1-n},
\ee
where $\omega_\Pi$ denotes the central character of $\RI^{\mathrm{aut}}(\Pi)$. We also denote by
\[
  \CG_\psi:  
    \bigsqcup_{\Pi\in  \CR_\varepsilon^{\mathrm{coti}} }  \RH^b_{\varepsilon}(\RI_\infty^{\mathrm{whi}}(\Pi))^\vee \rightarrow \bigsqcup_{\Pi\in  \CR_\varepsilon^{\mathrm{coti}} }  \RH^b_{\varepsilon}(\RI_\infty^{\mathrm{whi}}(\Pi))^\vee 
    \]
the fiberwise scalar multiplication map by \eqref{gpsi}. 
\begin{lemp}\label{lem32}
    The bundle map 
 \[
 \imath_\psi \circ \CG_\psi:  
    \bigsqcup_{\Pi\in  \CR_\varepsilon^{\mathrm{coti}} }  \RH^b_{\varepsilon}(\RI_\infty^{\mathrm{whi}}(\Pi))^\vee \rightarrow \bigsqcup_{\Pi\in  \CR_\varepsilon^{\mathrm{coti}} }  \RH^b_{\varepsilon}(\RI_\infty^{\mathrm{whi}}(\Pi))^\vee 
    \]
    over $\breve\ : \CR_\varepsilon^{\mathrm{coti}}\rightarrow \CR_\varepsilon^{\mathrm{coti}}$ is $\Aut(\C)$-equivariant.  
\end{lemp}
\begin{proof}
 In view of the commutative diagram \eqref{j}, this follows by Lemmas \ref{leminvf} and \ref{leminvf2}.
\end{proof}

Recall from \eqref{varpi} the $\Aut(\C)$-equivariant section \be\label{varpi22}
\varpi: \CR_\varepsilon^{\mathrm{coti}}\rightarrow  \bigsqcup_{\Pi\in  \CR_\varepsilon^{\mathrm{coti}} } \left(\RH^b_\varepsilon(\RI_\infty^{\mathrm{whi}}(\Pi))^\vee\setminus \{0\}\right),
\ee
and  the Betti-Whittaker period 
    \[
   \Omega_\varepsilon(\Pi) := \Omega_\varepsilon(\Pi; \varpi, \kappa_\varepsilon):=\la \varpi(\Pi), \kappa_{\varepsilon,\Pi}\ra. 
    \]
Define another section 
\begin{eqnarray*}
\breve \varpi_\psi  :\CR_\varepsilon^{\mathrm{coti}}&\rightarrow  &\bigsqcup_{\Pi\in  \CR_\varepsilon^{\mathrm{coti}} } \left(\RH^b_\varepsilon(\RI_\infty^{\mathrm{whi}}(\Pi))^\vee\setminus \{0\}\right),\\
\breve \Pi&\mapsto &\mathscr G(\omega_{\Pi}, \psi)^{1-n}\cdot \imath_\psi(\varpi({\Pi})),
\end{eqnarray*}
    which is $\Aut(\C)$-equivariant by Lemmas \ref{lemaute0} and \ref{lem32}. 

Now by Corollary \ref{mvwkappa} we have that 
\begin{eqnarray*}
    && \Omega_\varepsilon(\breve \Pi; \breve \varpi_\psi,\kappa_\varepsilon)\\
    &=& \la \breve \varpi_\psi(\breve \Pi), \kappa_{\varepsilon,\breve \Pi}\ra\\
    &=& \la \mathscr G(\omega_{\Pi}, \psi)^{1-n}\cdot \imath_\psi(\varpi({\Pi})), \kappa_{\varepsilon,\breve \Pi}\ra\\
   &=& \mathscr G(\omega_{\Pi}, \psi)^{1-n}\cdot \la  \varpi({\Pi}), \imath_\psi(\kappa_{\varepsilon,\breve \Pi})\ra\\
   &=& \mathscr G(\omega_{\Pi}, \psi)^{1-n}\cdot \delta_n\cdot \la  \varpi({\Pi}), \kappa_{\varepsilon, \Pi}\ra\\
          &=& \mathscr G(\omega_{\Pi}, \psi)^{1-n}\cdot \delta_n\cdot \Omega_\varepsilon(\Pi;\varpi, \kappa_\varepsilon). 
\end{eqnarray*}
In view of Lemma \ref{sectionpi00}, this proves Theorem \ref{maintheorem}.

\section{Proof of Theorem \ref{mvwpiepsilon}}

 Let $\K=\R$ or $\C$, and let $\pi_n$ be a cohomological  generic irreducible Casselman-Wallach representation of $\GL_n(\K)$ with trivial coefficient system. Up to isomorphism, the representation $\pi_n$ is unique unless $\K=\R$ and $n$ is odd, in which case there are two such representations. Let $\psi_\K: \K \to\C^\times$ be a non-trivial unitary character. Similar to \eqref{imathinfty}, the MVW-involution on $\GL_n(\K)$ and on the Whittaker model (associated to $\psi_\K$) of $\pi_n$  induce an involution 
\[
\imath_{\psi_\K}: \RH^{b_{n,\K}}(\pi_n) \to \RH^{b_{n,\K}}(\pi_n),
\]
where $\RH^{b_{n,\K}}(\pi_n):=\RH_{\mathrm{ct}}^{b_{n,\K}}(\GL_n(\K)^0; \pi_n)$ and 
\[
b_{n,\K} := \begin{cases} \lfloor \frac{n^2}{4}\rfloor, & \quad\textrm{if } \K =\R;  \\
\frac{n(n-1)}{2},  &  \quad \textrm{if }\K=\C.
\end{cases}
\]
Theorem \ref{mvwpiepsilon} is a direct consequence of the following result.

\begin{thmp} \label{mvwpiK}
The map 
\[
\imath_{\psi_\K}: \RH^{b_{n,\K}}(\pi_n) \to \RH^{b_{n,\K}}(\pi_n)
\]
is equal to the multiplication by 
\be\label{defdelta}
\delta_{n,\K}: =\begin{cases}(-1)^{\frac{m(m+1)}{2}}, & \quad \textrm{if  }\K=\R; \\
1, & \quad \textrm{if  }\K=\C,
\end{cases}
\ee
where $m:=\lfloor \frac{n-1}{2}\rfloor$.
\end{thmp}

The rest of this section is devoted to a proof of Theorem \ref{mvwpiK}. By twisting the sign character if necessary (only when $\K=\R$ and $n$ is odd), in the following subsections we assume  without loss of generality that $\pi_n$ has trivial central character.

\subsection{MVW-involution on the Lie algebra}

Let $K_n$ be the standard maximal compact subgroup of $\GL_n(\K)$, which is  $\RO(n)$ or $\RU(n)$ respectively when $\K=\R$ or $\C$. Its complexified Lie algebra  is equal to 
\[
\k_n:=\begin{cases}
    \{x\in \g\l_n(\C)\,:\, \textrm{$x$ is skew symmetric}\}, \quad &\textrm{if $\K=\R$};\\
    \g\l_n(\C), \quad &\textrm{if $\K=\C$}. 
    \end{cases}
\]
Let $\frak p_{n,\R} \subset \frak{gl}_n(\K)$ be the space of $n\times n$ symmetric or Hermitian matrices, respectively when $\K=\R$ or $\C$. Its complexification equals 
\[
\frak p_{n}:=\begin{cases}
    \{x\in \g\l_n(\C)\,:\, \textrm{$x$ is symmetric}\}, \quad &\textrm{if $\K=\R$};\\
    \g\l_n(\C), \quad &\textrm{if $\K=\C$},
    \end{cases}
\]
which is a representation of $K_n$ under the adjoint action. 

Define a Lie algebra automorphism 
\be\label{theta}
\theta :\g\l_n(\C)\rightarrow \g\l_n(\C),\qquad x\mapsto -x^t.
\ee

\begin{lemp}\label{unib}
   Up to conjugation by $\RO(n)$, there exists a unique Borel subalgebra of $\g\l_n(\C)$ that is $\theta$-stable.  
\end{lemp}
\begin{proof}
   This is known to experts and we sketch a proof for the convenience of the reader. 
A $\theta$-stable Borel subalgebra of $\mathfrak{gl}_n(\mathbb C)$ corresponds to a complete isotropic flag in $\C^n$ with respect to the standard symmetric bilinear form, i.e. a point of the full flag variety of $\RO_n(\C)$. Then the lemma follows by the Iwasawa decomposition for $\RO_n(\C)$.
\end{proof}

Let $\b_n$ be a $\theta$-stable Borel subalgebra of $\g\l_n(\C)$. Denote  by 
\be\label{defnn}
\n_n:=\textrm{ the nilpotent radical of $\b_n\cap \s\l_n(\C)$}.
\ee
Set
\[
\frak n'_{n}:=
    \n_n\cap \p_n, 
\]
which equals $\n_n$ when $\K=\C$. 
Note that 
\[
\dim \mathfrak{n}_n'=b_{n,\K},
\]
and the one-dimensional subspace $\wedge^{b_{n,\K}} \n'_n$ of $\wedge^{b_{n,\K}} \p_n$ generates an irreducible  $K_n$-subrepresentation, to be denoted by $\rho_n$. Moreover, the subrepresentation $\rho_n$  is independent of $\b_n$, has trivial central character, and has highest weight 
\[
\begin{cases}
   \left(n, n-2, \dots, n+2-2\left\lfloor \frac{n}{2}\right\rfloor\right), \quad &\textrm{if $\K=\R$};\\
    \left(n-1, n-3, \dots, 1-n\right), \quad &\textrm{if $\K=\C$}.
    \end{cases}  
\]

The one-dimensional subspace $\wedge^{b_{n,\K}} \n'_n$ of $\wedge^{b_{n,\K}} \p_n$ also generates an irreducible $\RO(n)$-subrepresentation, to be denoted by $\rho'_n$. It equals $\rho_n$ when $\K=\R$. When $\K=\C$, it is also independent of $\b_n$, has trivial central character, and has highest weight
\[
\left(2n-2, 2n-6,\dots, 2n+2-4\left\lfloor \frac{n}{2}\right\rfloor\right). 
\]

Note that as an irreducible representation of $K_n$, $\rho_n$ occurs in $\wedge^{b_{n,\K}} \p_n$ with multiplicity one, and likewise, as an irreducible representation of $\RO(n)$, $\rho'_n$ occurs in $\wedge^{b_{n,\K}} \p_n$ with multiplicity one.

The Lie algebra automorphism \eqref{theta} induces a linear automorphism 
\[
\theta :\frak p_n\rightarrow \frak p_n,
\]
which further induces a linear  automorphism 
\be\label{theta00003}
\theta :\wedge^{b_{n,\K}}\p_n\rightarrow \wedge^{b_{n,\K}} \frak p_n. 
\ee
The above map is equivariant with respect to the isomorphism
\be\label{bark}
K_n\rightarrow K_n, \quad g\mapsto \bar g\quad (\textrm{the entrywise complex conjugation}).
\ee
When $\K=\R$, the map \eqref{theta00003} is equal to the scalar multiplication by $(-1)^{b_{n,\K}}$. 

\begin{lemp}\label{xinK}
   The  restriction to $\rho'_n$ of the linear  automorphism 
   \[
\theta :\wedge^{b_{n,\K}}\p_n\rightarrow \wedge^{b_{n,\K}} \frak p_n 
\]
is equal to the scalar multiplication by
\[
\xi_{n,\K}:=\begin{cases}
    (-1)^{b_{n,\K}},\qquad&\textrm{if $\K=\R$};\\
    1,\qquad&\textrm{if $\K=\C$ and $n$ is odd};\\
    (-1)^{\frac{n}{2}},\qquad&\textrm{if $\K=\C$ and $n$ is even}.
\end{cases}
\]
\end{lemp}
\begin{proof}
This is clear when $\K=\R$. When $\K=\C$, we only need to calculate the determinant of the linear automorphism 
  \[
 \theta: \n_n \rightarrow \n_n,
  \]  
  where $\n_n$ is as in \eqref{defnn}. This can be done by counting the imaginary roots and the complex roots for $\GL_n(\R)$ with respect to a fundamental Cartan subgroup. We omit the details. 
\end{proof}

\subsection{A characterization of $\delta_{n,\K}$}

Denote by $\tau_n\subset \pi_n$ the minimal $K_n$-type in the sense of Vogan (see \cite{VZ84}). It is an irreducible representation of $K_n$ which is isomorphic to $\rho_n$ and occurs in $\pi_n$ with multiplicity one. By \cite[Proposition 9.4.3]{W88} and \cite[Proposition 3.2]{VZ84}, we have identifications 
\begin{equation}\label{vz84}
 \RH^b_{\rm ct}(\GL_n(\K);  \pi_n)= \Hom_{K_n}( \wedge^{b_{n,\K}}\p_n, \pi_n)= \Hom_{K_n}( \rho_n, \pi_n)=\Hom_{K_n}( \rho_n, \tau_n)
\end{equation}
of one-dimensional vector spaces. 


As before, $\pi_n$ is realized as its  Whittaker model (associated to $\psi_\K$) and we have a linear automorphism 
\be\label{equ2}
\pi_n\rightarrow \pi_n, \quad f\mapsto w_n.\breve f.
\ee
This map is also equivariant with respect to the isomorphism
\eqref{bark}.

\begin{lemp}\label{lemdelta}
  There exists a unique number $\delta'_{n,\K}\in \C^\times$ such that   the diagram 
    \be\label{comta1}
\begin{CD}
  \wedge^{b_{n,\K}}\frak p_{n} @>\phi  >>\pi_n\\
    @V\delta'_{n,\K}\cdot \theta  VV @VV  f\mapsto w_n.\breve f V\\
    \wedge^{b_{n,\K}}\frak p_{n}@>\phi \ >>\pi_n\\
\end{CD}
\ee
commutes for all $\phi\in \Hom_{K_n^0}( \wedge^{b_{n,\K}}\frak p_n, \pi_n)$.
\end{lemp}
\begin{proof}
  Fix a generator $\phi_0$ of the one-dimensional space $\Hom_{K_n}( \rho_n, \pi_n)$. Then by the equivariance of \eqref{theta00003} and \eqref{equ2}  with respect to the isomorphism
\eqref{bark}, there exists a unique number $\delta'_{n,\K}\in \C^\times$ such that   the diagram 
    \be\label{comta2}
\begin{CD}
  \rho_n @>\phi_0  >>\pi_n\\
    @V\delta'_{n,\K}\cdot \theta  VV @VV  f\mapsto w_n.\breve f V\\
   \rho_n@>\phi_0 \ >>\pi_n\\
\end{CD}
\ee
commutes.

When $\K=\R$ and $n$ is even, we have a $K_n^0$-stable decomposition (see \cite[Chapter V]{We97}) \[
\rho_n|_{K_n^0}=\rho_n^+\oplus \rho_n^-,
\]
 where  $\rho_n^+$ and $\rho_n^-$ are irreducible subrepresentations of $\rho_n|_{K_n^0}$ which are not isomorphic to each other. Then 
 \[
\phi_0|_{\rho_n^+}\in \Hom_{K_n^0}( \rho_n^+, \pi_n)\qquad\textrm{and}\qquad \phi_0|_{\rho_n^-}\in  \Hom_{K_n^0}( \rho_n^-, \pi_n) 
 \]
 are generators in the one-dimensional spaces, and  the diagrams 
    \[
\begin{CD}
  \rho_n^+ @>\phi_0|_{\rho_n^+}  >>\pi_n\\
    @V\delta'_{n,\K}\cdot \theta  VV @VV  f\mapsto w_n.\breve f V\\
   \rho_n^+@>\phi_0|_{\rho_n^+} \ >>\pi_n\\
\end{CD} \qquad\qquad 
\textrm{and}\qquad \qquad 
\begin{CD}
  \rho_n^- @>\phi_0|_{\rho_n^-}  >>\pi_n\\
    @V\delta'_{n,\K}\cdot \theta  VV @VV  f\mapsto w_n.\breve f V\\
   \rho_n^-@>\phi_0|_{\rho_n^-} \ >>\pi_n\\
\end{CD} 
\]
commute. This implies that the diagram \eqref{comta1} commutes. In all other cases, the commutativity of the diagram \eqref{comta1} directly follows from the commutativity of the diagram \eqref{comta2}. This proves the lemma. 
\end{proof}

Let $\delta'_{n,\K}\in \C^\times$ be as in Lemma \ref{lemdelta}. Since $w_n\in K_n^0$ and $w_n^2$ acts trivially on $\wedge^{b_{n,\K}}\frak p_{n} $, Lemma  \ref{lemdelta} implies that the diagram 
\[
\begin{CD}
  \wedge^{b_{n,\K}}\frak p_{n} @>\phi  >>\pi_n\\
    @V u\mapsto \delta'_{n,\K}\cdot w_n.(\theta(u)) VV @VV   \breve \  V\\
    \wedge^{b_{n,\K}}\frak p_{n}@>\phi \ >>\pi_n\\
\end{CD}
\]
commutes for all $\phi\in \Hom_{K_n^0}( \wedge^{b_{n,\K}}\frak p_{n}, \pi_n)$. 
Therefore, in order to prove Theorem \ref{mvwpiK}, it remains to show that 
\be\label{eqdelta}
\delta'_{n,\K}=\delta_{n,\K},
\ee
where $\delta_{n,\K}$ is defined in \eqref{defdelta}.

\subsection{Rankin-Selberg convolutions}
The equality \eqref{eqdelta} clearly holds when $n=1$. We prove it for general $n$ by induction. Thus we assume that $n\geq 2$ and \eqref{eqdelta} holds  for $n-1$. 

As usual, $\GL_{n-1}(\K)$ is identified as a subgroup of $\GL_n(\K)$ via the embedding
\[
g\mapsto \begin{bmatrix} g &0 \\ 0 & 1\end{bmatrix}.
\]
Let $\pi_{n-1}$ be the cohomological  generic irreducible Casselman-Wallach representation of $\GL_{n-1}(\K)$ with trivial coefficient system and trivial central character. We realize it as the Whittaker model associated to $\psi_\K^{-1}$, and then define an involution
\[
\breve \ : \pi_{n-1}\rightarrow \pi_{n-1}
\]
as before. 

\begin{lemp} \label{lemFE}
There exists a nonzero $\GL_{n-1}(\K)$-invariant continuous bilinear map 
\[
\oZ^\circ: \pi_n\times \pi_{n-1}\rightarrow \C
\]
such that 
 the diagram 
\begin{equation}\label{comta}
\begin{CD}
  \pi_n\times \pi_{n-1} @>\oZ^\circ  >>\C \\
    @V (f,f')\mapsto (w_n.\breve f,\, w_{n-1}.\breve f') VV @VV   \textrm{multiplication by }\, \varepsilon\left(\frac{1}{2}, \pi_n\times \pi_{n-1}, \psi_\K\right) V\\
   \pi_n\times \pi_{n-1} @>  \oZ^\circ>>\C \\
\end{CD}
\end{equation}
commutes, 
where
\[
 \varepsilon\left(\frac{1}{2}, \pi_n\times \pi_{n-1}, \psi_\K\right) = \begin{cases}1 , & \quad \textrm{if  $\K=\R$ and  $n$ is odd}; \\
 (-1)^{\lfloor n/2\rfloor}, & \quad \textrm{otherwise.}
 \end{cases}
\]
\end{lemp}

\begin{proof}
We choose  $\oZ^\circ$ to be the normalized Rankin-Selberg integral (see \cite[Section 1.4.1]{JLS25} and \cite{J09} for the precise definition). 
Note that $\pi_n$ and $\pi_{n-1}$ are both self-dual, and they have trivial central characters by assumption. 
The lemma follows easily from the functional equation of Rankin-Selberg integrals (\cite[Theorem 2.1]{J09}) and a straightforward computation of local $\varepsilon$-factors. 
\end{proof}

 Let $\RZ^\circ$ be as in Lemma \ref{lemFE}. Denote by $\tau_n'$ the $\RO(n)$-subrepresentation of $\tau_n$ that corresponds to the $\RO(n)$-subrepresentation $\rho_n'$ of $\rho_n$. Similarly define $\tau_{n-1}'\subset \tau_{n-1}$.

\begin{lemp} \label{nonvh}
It holds that
   \[
  \RZ^\circ|_{\tau_n'\times \tau_{n-1}'}\neq 0.
  \]
\end{lemp}

\begin{proof}   
If $\K=\R$, then it is proved in   \cite[Proposition 4.1]{Sun} that  
\be\label{nonzero1}
\RZ^\circ|_{\tau_n\times \tau_{n-1}}\neq 0,
\ee
and hence the lemma follows. 

Now we assume that $\K=\C$.  
The same proof as that of \cite[Proposition 4.1]{Sun} shows that \eqref{nonzero1} also holds in this case. Then  by \cite[Lemmas 2.1 and 2.2]{Sun}, \eqref{nonzero1} implies that
  \[
  \RZ^\circ|_{\tau_n'\times \tau_{n-1}'}\neq 0.
  \]
  This finishes the proof of the lemma. 
\end{proof}

Fix isomorphisms   
\[
0\neq \phi_n\in \Hom_{K_n}(\rho_n, \tau_n)\qquad\textrm{and}\qquad 0\neq \phi_{n-1}\in \Hom_{K_{n-1}}(\rho_{n-1}, \tau_{n-1}). 
\]
Combining the diagrams \eqref{comta1} and \eqref{comta}, we see that the diagram
\[
\begin{CD}
 \rho'_n\otimes \rho'_{n-1}  @>\phi_n|_{\rho'_n}\otimes\phi_{n-1}|_{\rho'_{n-1}}>>\tau'_n\otimes \tau'_{n-1} @>\oZ^\circ  >>\C \\
  @V(\delta_{n,\K}'\theta)\otimes (\delta_{n-1,\K}'\theta) VV  @V f\otimes f'\mapsto (w_n.\breve f)\otimes (w_{n-1}.\breve f') VV @VV  (\,\cdot\,) \cdot  \varepsilon\left(\frac{1}{2}, \pi_n\times \pi_{n-1}, \psi_\K\right) V\\
 \rho'_n\otimes \rho'_{n-1}  @>\phi_n|_{\rho'_n}\otimes\phi_{n-1}|_{\rho'_{n-1}}>>  \tau'_n\otimes \tau'_{n-1} @>  \oZ^\circ>>\C \\
\end{CD}
\]
commutes. 
By Lemmas \ref{xinK} and \ref{nonvh}, we have that
\[
\delta_{n,\K}'\xi_{n,\K}\cdot\delta_{n-1,\K}'\xi_{n-1,\K}=\varepsilon\left(\frac{1}{2}, \pi_n\times \pi_{n-1}, \psi_\K\right).
\]
It is straightforward to check that $\delta_{n-1,\K}'=\delta_{n-1,\K}$ implies that $\delta_{n,\K}'=\delta_{n,\K}$. This completes the proof of Theorem \ref{mvwpiK} by induction on $n$.

	\section*{Acknowledgments}
D. Liu was supported in part by National Key R \& D Program of China No. 2022YFA1005300 and National Natural Science Foundation of China No. 12526208.  B. Sun was supported in part by National Key R \& D Program of China No. 2022YFA1005300  and New Cornerstone Science Foundation. The authors thank Shih-Yu Chen for the communication about the work \cite{CCDR}.
	

\end{document}